\documentclass[11pt,a4paper]{amsart}
\usepackage{mathrsfs}
\usepackage{syntonly}
\usepackage{amsmath}
\usepackage{amsthm}
\usepackage{amsfonts}
\usepackage{amssymb}
\usepackage{latexsym}
\usepackage{amscd,amssymb,amsopn,amsmath,amsthm,graphics,amsfonts,mathrsfs,accents,enumerate,verbatim,calc}
\usepackage[dvips]{graphicx}
\usepackage[colorlinks=true,linkcolor=blue,citecolor=blue]{hyperref}
\usepackage[all]{xy}

\date{}
\allowdisplaybreaks[4] \footskip=15pt
\renewcommand{\uppercasenonmath}[1]{}

\numberwithin{equation}{section} \theoremstyle{plain}
\newtheorem{lem}{Lemma}[section]
\newtheorem{cor}[lem]{Corollary}
\newtheorem{prop}[lem]{Proposition}
\newtheorem{thm}[lem]{Theorem}

\newtheorem{cond}[lem]{Condition}
\newtheorem{definition}[lem]{Definition}
\newtheorem{Ex}[lem]{Example}
\newtheorem{Quest}[lem]{Question}
\newtheorem{Property}[lem]{Property}
\newtheorem{Properties}[lem]{Properties}
\newtheorem{Subprops}{}[lem]
\newtheorem{Para}[lem]{}

\newtheorem{fact}[lem]{Fact}

\newtheorem{remark}[lem]{Remark}
\newtheorem{rem}[lem]{Remark}

\def\blue{\color{blue}}

\newenvironment{ex}{\begin{Ex}\rm}{\end{Ex}}

\newtheorem*{ack*}{ACKNOWLEDGEMENTS}

\newcommand{\pf}{\noindent\begin {proof}}
\newcommand{\epf}{\end{proof}}

\newcommand{\X}{\mathcal{X}}

\newcommand{\C}{\mathcal{C}}

\newcommand{\A}{\mathcal{F}}

\newcommand{\B}{\mathcal{C}}

\newcommand{\h}{{\rm H}}

\begin{document}
\begin{center}
{\Large  \bf  Model structure arising from one hereditary complete cotorsion pair on extriangulated categories}

\vspace{0.5cm}  Jiangsheng Hu$^{a}$, Dongdong Zhang$^{b}$\footnote{Corresponding authors. Jiangsheng Hu is supported by the National Natural Science Foundation of China (Grant Nos. 12571035 and 12171206) and Jiangsu 333 Project. Dongdong Zhang is supported by National Natural Science Foundation of China (Grant No. 12571042). Pu Zhang is supported by National Natural Science Foundation of China (Grant No. 12131015) and by Natural Science Foundation of Shanghai (Grant No. 23ZR1435100). Panyue Zhou is supported by the National Natural Science Foundation of China (Grant No. 12371034) and by the Scientific Research Fund of Hunan Provincial Education Department (Grant No. 24A0221).}, Pu Zhang$^{c{\blue *}}$ and Panyue Zhou$^{d}$ \\
\medskip

\hspace{-4mm}$^{a}$School of Mathematics, Hangzhou Normal University, Hangzhou 311121, P. R. China\\[1mm]
 $^b$School of Mathematical Sciences, Zhejiang Normal University,
\small Jinhua 321004, P. R. China\\[1mm]
$^c$School of Mathematical Sciences, Shanghai Jiao Tong University,
\small  Shanghai 200240, P. R. China\\[1mm]
$^d$School of Mathematics and Statistics, Changsha University of Science and Technology, \\[1mm]
Changsha 410114, P. R. China\\[1mm]
E-mails: hujs@hznu.edu.cn, zdd@zjnu.cn, pzhang$\symbol{64}$sjtu.edu.cn and panyuezhou@163.com \\
\end{center}

\bigskip
\centerline { \bf  Abstract}
\medskip
\leftskip10truemm \rightskip10truemm
\noindent Hovey's correspondence between model structures and cotorsion pairs in the setting of abelian categories, has been generalized by Nakaoka-Palu, using two cotorsion pairs, to the setting of
weakly idempotent complete extriangulated categories, and the aim of the paper is to give
an analogous correspondence using one (hereditary) cotorsion pair generalizing in this setting
work of Beligiannis-Reiten and Cui, Lu and Zhang. Furthermore, we provide methods to construct model structures from silting objects in weakly idempotent complete extriangulated categories and co-$t$-structures on triangulated categories.
\\[2mm]
{\bf Keywords:} extriangulated category; cotorsion pair; model structure; silting object; co-$t$-structure\\
{\bf 2020 Mathematics Subject Classification:} 18N40; 16D90; 16E30; 16E65

\leftskip0truemm \rightskip0truemm
\section {\bf Introduction}
The Hovey correspondence \cite{HCc, Gillespie,Yang} is a tool for constructing model structures on abelian categories, exact categories or triangulated categories.
The notion of an extriangulated category was introduced in \cite{NP} as a common generalization of exact categories and triangulated categories, and Hovey's correspondence has been extended by Nakaoka and Palu in this context, where the corresponding homotopy category is proved to be a triangulated category (\cite[Theorem 6.20]{NP}; see also \cite[Theorem 6.34]{G}).

\vskip5pt

The Hovey correspondence needs two complete cotorsion pairs. Beligiannis and  Reiten \cite[VIII, 4.2, 4.13]{Beli-Reiten}
establish a one-to-one correspondence between weakly projective model structures (or called the $\omega$-model structures)
on abelian category $\mathcal{A}$ and hereditary complete cotorsion pairs
with the core $\omega$ contravariantly finite in $\mathcal{A}$.
This involves only one hereditary complete cotorsion pair.
Recently, it has been extended by Cui, Lu and Zhang \cite{CLZ}
to weakly idempotent complete exact categories $\mathcal{E}$. These weakly projective model structures are not necessarily exact.   Recall  that a model structure
on an exact category is {\it exact} (\cite[3.1]{Gillespie}), if cofibrations are precisely inflations with cofibrant cokernel, and fibrations are precisely deflations with fibrant kernel. The
two approaches get the same result if and only if $\mathcal{E}$ has enough projective objects and the model
structure is {\it projective} in the sense of Gillespie (\cite[4.5]{Gillespie}), i.e., it is exact and each object is fibrant.

\vskip5pt

The aim of this paper is to extend the Beligiannis-Reiten correspondence to weakly idempotent
complete extriangulated categories. This of course heavily involves the applications of properties of extriangulated categories.

\vskip5pt

To state the main results, we recall some definitions.
An additive category $\mathcal{A}$ is  {\it weakly idempotent complete} if every retraction has
a kernel, or equivalently, if every section has a cokernel \cite[Definition 7.2]{Buh}.
Assume that $\mathcal{B} = (\mathcal{B}, \mathbb{E}, \mathfrak{s})$ is a weakly idempotent complete extriangulated category (see \cite[Definition 2.10]{NP}).
A sequence $\xymatrix@C=0.6cm{A \ar[r]^{x} & B \ar[r]^{y} & C }$ in $\mathcal{B}$ is a \emph{conflation} if it realizes some $\mathbb{E}$-extension $\delta\in\mathbb{E}(C, A)$.
The conflation is usually written as $\xymatrix{A\ar[r]^x&B\ar[r]^{y}&C\ar@{-->}[r]^{\delta}&.}$
In this case, $x$ is called an {\it inflation}, $y$ a {\it deflation}, $C$ the \emph{cone} of $x$ which is denoted by ${\rm Cone}(x)$, and $A$ the \emph{cocone} of $y$ which is denoted by ${\rm CoCone}(y)$.

\vskip5pt Let $\mathcal{X}$ and $\mathcal{Y}$ be full subcategories of $\mathcal{B}$
which are closed under direct summands and isomorphisms. Put $\omega:=\mathcal{X}\cap \mathcal{Y}$. Define three classes of morphisms in $\mathcal{B}$ as follows.

\vskip5pt

\begin{center}${\rm CoFib}_{\omega}=\{f: A\rightarrow B\mid f $ is an inflation with ${\rm Cone}(f)\in\mathcal{X}\}$.
\vspace{2mm}

${\rm Fib}_{\omega}=\{f: A\rightarrow B\mid \mathcal{B}(W, f): \mathcal{B}(W, A)\rightarrow \mathcal{B}(W, B)$ is surjective for any $W\in\omega\}$. \qquad (1.1)
\vspace{2mm}

${\rm Weq}_{\omega}=\{f: A\rightarrow B\mid $ there is a deflation $(f, t): A\oplus W\rightarrow B$

 such that $W\in\omega$ and ${\rm CoCone}(f, t)\in\mathcal{Y}\}$. \end{center}

%
%
%
%
%
%
\vspace{1mm}


\begin{thm}\label{thmA} Let $\mathcal{B}$ be a weakly idempotent complete extriangulated category, let $\mathcal{X}$ and $\mathcal{Y}$ be full additive subcategories of $\mathcal{B}$
which are closed under direct summands and isomorphisms, and $\omega=\mathcal{X}\cap \mathcal{Y}$. Then $({\rm CoFib}_\omega, {\rm Fib}_\omega, {\rm Weq}_\omega)$ as defined in {\rm(1.1)} is a model structure on $\mathcal{B}$ if and only if $(\mathcal{X, Y})$ is a hereditary complete cotorsion pair in $\mathcal{B}$, and $\omega$ is contravariantly finite in $\mathcal{B}$.

\vskip5pt

If this is the case, then
\begin{center}
 ${\rm TCoFib}_{\omega}=\{f: A\rightarrow B\mid f$ is a splitting inflation with ${\rm Cone}(f)\in\omega\}$,
\vspace{2mm}

${\rm TFib}_{\omega}=\{f:A\rightarrow B\mid f$ is a deflation with ${\rm CoCone}(f)\in\mathcal{Y}\}$;
\end{center}
\noindent the class $\mathcal{C}_\omega$ of cofibrant objects is $\mathcal{X}$, the class $\mathcal{F}_\omega$ of fibrant objects is $\mathcal{B}$, the class $\mathcal{W}_\omega$ of trivial objects is $\mathcal{Y}$; and the homotopy category ${\rm Ho}(\mathcal{B})$ is equivalent to the additive quotient $\mathcal{X}/\omega$.
\end{thm}

For a full additive subcategory $\mathcal{U}$ of an additive category $\mathcal{A}$, recall that the quotient category $\mathcal{A}/\mathcal{U}$ has the same objects as $\mathcal{A}$, and
$$\textrm{Hom}_{\mathcal{A}/\mathcal{U}}(X,Y)=\textrm{Hom}_{\mathcal{A}}(X,Y)/\textrm{Hom}_{\mathcal{A}}(X,\mathcal{U},Y)$$
where $\textrm{Hom}_{\mathcal{A}}(X,\mathcal{U},Y)$ is the subgroup of $\textrm{Hom}_{\mathcal{A}}(X,Y)$ consisting of those morphisms which factor through an object in $\mathcal{U}$. Then $\mathcal{A}/\mathcal{U}$ is an additive category.

\vskip5pt

The model structure
given in Theorem \ref{thmA}  will be called the \emph{$\omega$-model structure}. The originality of Theorem \ref{thmA} is due to Beligiannis and Reiten \cite{Beli-Reiten} for abelian categories.
If $\mathcal{B}$ is a weakly idempotent complete exact category, Theorem \ref{thmA} recovers the $\omega$-model structure in \cite[Theorem 1.1]{CLZ}.

\vskip5pt

Thanks to the bijection between bounded hereditary cotorsion pairs
and silting subcategories in extriangulated categories (\cite[Theorem 5.7]{AT}),  Theorem \ref{thmA} yields that any silting object in a weakly idempotent complete extriangulated category induces a model structure (see Corollary \ref{cor:4.2}). Some examples  are given to illustrate this result (see Examples \ref{ex:4.1}-\ref{ex:4.3}).

\vskip5pt

A model structure on a weakly idempotent complete extriangulated category is \emph{weakly projective} if cofibrations are precisely inflations with cofibrant cone, trivial fibrations are precisely deflations with trivially fibrant cocone, and each object is fibrant.
See Proposition \ref{prop4.4} for equivalent characterizations of a weakly projective model structure. As in weakly idempotent complete exact categories (\cite[Theorem 1.3]{CLZ}),
the $\omega$-model structures on a weakly idempotent complete extriangulated category are precisely the weakly projective model structures.

\begin{thm}\label{thm4.6} {\rm(The Beligiannis-Reiten correspondence)} Let $\mathcal{B}$ be a weakly idempotent complete extriangulated category, $\Omega$ the class of hereditary complete cotorsion pairs $(\mathcal{X}, \mathcal{Y})$ with $\omega=\mathcal{X}\cap\mathcal{Y}$  contravariantly finite, and $\Gamma$ the class of weakly projective model structures on $\mathcal{B}$. Then the map
$$\Phi: (\mathcal{X}, \mathcal{Y})\mapsto ({\rm CoFib}_\omega, {\rm Fib}_\omega, {\rm Weq}_\omega)$$
as given in $(1.1)$, is a bijection between $\Omega$ and $\Gamma$, with the inverse given by
$$\Psi:({\rm CoFib}, {\rm Fib}, {\rm Weq})\mapsto (\mathcal{C}, {\rm T}\mathcal{F})$$
where $\mathcal{C}$ and ${\rm T}\mathcal{F}$ are respectively the class of cofibrant objects and the class of trivially fibrant objects, of the model structure $({\rm CoFib}, {\rm Fib}, {\rm Weq})$.
\end{thm}

Theorem \ref{thm4.6} is discovered by Beligiannis and Reiten for abelian categories in \cite[Theorem 4.6]{Beli-Reiten}. It is generalized to weakly idempotent complete exact category in \cite[Theorem 1.3]{CLZ}. Furthermore, specializing the weakly idempotent complete extriangulated category $\mathcal{B}$ to a triangulated category, Theorem \ref{thm4.6} yields a one-to-one correspondence between weakly projective model structures on $\mathcal{B}$ and co-$t$-structures with the coheart contravariantly finite in $\mathcal{B}$ (see Corollary \ref{corB}).

\vskip5pt

The paper is organized  as follows. Section 2 recalls necessary preliminaries on extriangulated categories, cotorsion pairs, model structures and their homotopy categories.
Section 3 contributes to showing Theorem \ref{thmA}. In Section 4, weakly projective model structures are characterized, and Theorem \ref{thm4.6} is proved. Some corollaries of Theorem \ref{thmA} are given in the final section.

\section{\bf Preliminaries}
Throughout $\mathcal{B}$ is an additive category. All subcategories are full and closed under isomorphisms and direct summands.

\subsection{Extriangulated categories} Suppose throughout that $\mathbb{E}: \mathcal{B}^{op}\times \mathcal{B}\rightarrow {\rm Ab}$ is an additive bifunctor,  where ${\rm Ab}$ is the category of abelian groups. For any objects $A, C\in\mathcal{B}$, an element $\delta\in \mathbb{E}(C,A)$ is called {\it an $\mathbb{E}$-extension}. The zero element $0\in\mathbb{E}(C, A)$ is called {\it the split $\mathbb{E}$-extension}.
For any $a\in\mathcal{B}(A, A')$ and $c\in\mathcal{B}(C', C)$, since $\mathbb E$ is a bifunctor, one has $\mathbb{E}$-extensions
$\mathbb{E}(C, a)(\delta)\in\mathbb{E}(C, A')$ and $\mathbb{E}(c, A)(\delta)\in\mathbb{E}(C', A)$, which is simply denoted by $a_*\delta$ and $c^*\delta$, respectively,
and one has the equality in $\mathbb{E}(C', A'):$ $$\mathbb{E}(c, a)(\delta)=c^*a_*\delta=a_*c^*\delta.$$

\vskip5pt

Let $\delta\in\mathbb{E}(C, A)$ and $\delta'\in\mathbb{E}(C', A')$ be $\mathbb{E}$-extensions. {\it A morphism} $(a, c): \delta\rightarrow \delta'$ of $\mathbb{E}$-extensions is a pair of morphisms $a\in\mathcal{B}(A, A')$ and $c\in\mathcal{B}(C, C')$ in $\mathcal{B}$ such that $a_*\delta=c^*\delta'.$

\vskip5pt

Let $\xymatrix{C\ar[r]^{\iota_C\ \ \ }&C\oplus C'&C'\ar[l]_{\qquad\iota_{C'}}}$
 and
 $\xymatrix{A&A\oplus A'\ar[l]_{ p_A\;\;}\ar[r]^{\quad p_{A'} }&A'}$
be coproduct and product in $\mathcal{B}$, respectively. By the additivity of $\mathbb{E}$ one has a natural isomorphism
\begin{center} $\mathbb{E}(C\oplus C', A\oplus A')\simeq\mathbb{E}(C, A)\oplus\mathbb{E}(C, A')
\oplus\mathbb{E}(C', A)\oplus\mathbb{E}(C', A')$.\end{center}
Let $\delta\oplus \delta'\in\mathbb{E}(C\oplus C', A\oplus A')$ be the element corresponding to $(\delta, 0, 0, \delta')$ through this isomorphism. This is the unique element which satisfies
\begin{center}$\mathbb{E}(\iota_C, p_A)(\delta\oplus \delta')=\delta,\ \mathbb{E}(\iota_C, p_{A'})(\delta\oplus \delta')=0, \ \mathbb{E}(\iota_{C'}, p_A)(\delta\oplus \delta')=0,\ \mathbb{E}(\iota_{C'}, p_{A'})(\delta\oplus \delta')=\delta'$.\end{center}

 Recall that two sequences
 $\xymatrix{A\ar[r]^x&B\ar[r]^y&C}$ and $\xymatrix{A\ar[r]^{x'}&B'\ar[r]^{y'}&C}$ of morphisms in $\mathcal{B}$
 are {\it equivalent}, if there is an isomorphism $b\in\mathcal{B}(B, B')$ such that $x' = b\circ x, \ y'\circ b = y.$
The corresponding equivalence class is denoted by
$\xymatrix{[A\ar[r]^x&B\ar[r]^y&C]}$.
For any $A, C\in\mathcal{B}$, Put
 $$0=[A\xrightarrow{~\tiny\begin{bmatrix}1\\0\end{bmatrix}~}A\oplus C\xrightarrow{\tiny\begin{bmatrix}0&1\end{bmatrix}}C]$$
and
\begin{center} $\xymatrix{[A\ar[r]^x&B\ar[r]^y&C]}\oplus$$\xymatrix{[A'\ar[r]^{x'}&B'\ar[r]^{y'}&C']}=$$\xymatrix{[A\oplus A'\ar[r]^{x\oplus x'}&B\oplus B'\ar[r]^{y\oplus y'}&C\oplus C'].}$\end{center}

 \begin{definition}\label{def1}{\rm (\cite[Definitions 2.9, 2.10]{NP}) \ A {\it realization} $\mathfrak{s}$ of $\mathbb{E}$ is a correspondence, which  assigns any $\mathbb{E}$-extension $\delta\in \mathbb{E}(C, A)$  for any object $A, C\in \mathbb {E}$ to
an equivalence class $$\mathfrak{s}(\delta)=\xymatrix@C=0.8cm{[A\ar[r]^x
 &B\ar[r]^y&C]}$$
satisfying the following condition $(\star)$. In this case, one says that the sequence $\xymatrix{A\ar[r]^x&B\ar[r]^y&C}$ realizes $\delta$.

\vskip5pt

$(\star)$ \ \ Let $\delta\in\mathbb{E}(C, A)$ and $\delta'\in\mathbb{E}(C', A')$ be any pair of $\mathbb{E}$-extensions, with
 \begin{center} $\mathfrak{s}(\delta)=$$\xymatrix{[A\ar[r]^x &B\ar[r]^y&C]}$ and $\mathfrak{s}(\delta')=$$\xymatrix{[A'\ar[r]^{x'} &B'\ar[r]^{y'}&C'].}$\end{center} Then for any morphism $(a, c): \delta\rightarrow \delta'$, there exists $b\in\mathcal{B}(B, B')$ such that the diagram commutes:
 $$\xymatrix@R = 0.5cm{A\ar[r]^{x}\ar[d]_{a}&B\ar[r]^{y}\ar[d]_{b}&C\ar[d]_{c}\\
 A'\ar[r]^{x'}&B'\ar[r]^{y'}&C'.}$$
In this situation, one says that the triple $(a, b, c)$ realizes $(a, c)$.

\vskip5pt

A realization $\mathfrak{s}$ of $\mathbb{E}$ is {\it additive}, if it satisfies the following conditions:

{\rm (i)} \ For any $A, C\in\mathcal{B}$, the split $\mathbb{E}$-extension $0\in\mathbb{E}(C, A)$ satisfies $\mathfrak{s}(0)=0.$

{\rm (ii)} \ For any pair of $\mathbb{E}$-extensions $\delta\in\mathbb{E}(C, A)$ and $\delta'\in\mathbb{E}(C', A')$, one has \begin{center}$\mathfrak{s}(\delta\oplus\delta')=\mathfrak{s}(\delta)\oplus\mathfrak{s}(\delta')$.\end{center}}
 \end{definition}

 \begin{definition} {\rm (\cite[Definition 2.12]{NP}) A triplet $(\mathcal{B}, \mathbb{E}, \mathfrak{s})$ is an {\it extriangulated category} if it satisfies the following conditions.

{\rm(ET1)} \ $\mathbb{E}: \mathcal{B}^{\rm op}\times \mathcal{B}\rightarrow \rm{Ab}$ is an additive bifunctor.

{\rm (ET2)} \ $\mathfrak{s}$ is an additive realization of $\mathbb{E}$.

{\rm (ET3)} \ Let $\delta\in\mathbb{E}(C, A)$ and $\delta'\in\mathbb{E}(C', A')$ be $\mathbb{E}$-extensions, realized as
 \begin{center} $\mathfrak{s}(\delta)=$$\xymatrix{[A\ar[r]^x &B\ar[r]^y&C]}$ and \ $\mathfrak{s}(\delta')=$$\xymatrix{[A'\ar[r]^{x'} &B'\ar[r]^{y'}&C'].}$\end{center} For any commutative square
 $$\xymatrix@R=0.5cm{A\ar[r]^{x}\ar[d]_{a}&B\ar[r]^{y}\ar[d]_{b}&C\\
 A'\ar[r]^{x'}&B'\ar[r]^{y'}&C'}$$
 in $\mathcal{B}$, there exists a  morphism $(a, c): \delta\rightarrow \delta'$ such that $c\circ y=y'\circ b$. (Then  one has $a_*\delta = c^*\delta'.$)
\vspace{1mm}

{\rm (ET3)}$^{\rm op}$ \ The dual of {\rm (ET3)}.
\vspace{1mm}

{\rm (ET4)} \ Let $\delta\in\mathbb{E}(D,A)$ and $\delta'\in\mathbb{E}(F, B)$ be $\mathbb{E}$-extensions realized by
 \begin{center} $\xymatrix{A\ar[r]^f&B\ar[r]^{f'}&D}$ and $\xymatrix{B\ar[r]^g&C\ar[r]^{g'}&F}$\end{center}
 respectively. Then there exist an object $E\in\mathcal{C}$, a commutative diagram
 $$\xymatrix@R=0.7cm{A\ar[r]^f\ar@{=}[d]&B\ar[r]^{f'}\ar[d]_g&D\ar[d]^d\\
A\ar[r]^h&C\ar[r]^{h'}\ar[d]_{g'}&E\ar[d]^e\\
&F\ar@{=}[r]&F
}$$
in $\mathcal{B}$, and an $\mathbb{E}$-extension $\delta^{''}\in\mathbb{E}(E, A)$ realized by $\xymatrix{A\ar[r]^h&C\ar[r]^{h'}&E,}$
satisfying the following compatibilities:

{\rm (i)} \ \ $\xymatrix{D\ar[r]^d&E\ar[r]^{e}&F}$ realizes $f'_*\delta'$,

{\rm (ii)} \ \ $d^*\delta^{''}=\delta$,

{\rm (iii)} \ \ $f_*\delta^{''}=e^*\delta'$.
\vspace{1mm}

{\rm (ET4)}$^{\rm op}$ \ The dual of {\rm (ET4)}.
} \end{definition}

\begin{rem} $(1)$ \ Exact categories and triangulated categories are extriangulated {\rm (\cite[Example 2.13]{NP})}. More precisely, let $\mathcal{B}$ be a triangulated category with shift functor $[1]$. Put $\mathbb{E}(-,-):=\mathcal{B}(-, -[1])$.
For any $\delta\in\mathbb{E}(C, A)=\mathcal{B}(C, A[1])$, then $\delta$ is embedded into a unique distinguished triangle
$\xymatrix{A\ar[r]^x&B\ar[r]^y&C\ar[r]^\delta&A[1],}$ up to an isomorphism of triangles. Define $\mathfrak{s}(\delta)=\xymatrix{[A\ar[r]^x&B\ar[r]^y&C]}.$
Then $(\mathcal{B}, \mathbb{E}, \mathfrak{s})$ becomes an extriangulated category. Let $\mathcal{B}$ be an exact category. Put $\mathbb{E}(-,-):={\rm Ext}^1_\mathcal B(-, -)$ and $\mathfrak{s}(\delta)=\delta$ for any $\delta\in\mathbb{E}(C, A)$. Then $(\mathcal{B}, \mathbb{E}, \mathfrak{s})$ becomes an extriangulated category.

$(2)$ There exist extriangulated categories which are neither exact categories nor triangulated categories $($see \cite[Proposition 3.30]{NP}, \cite[Remark 4.13]{ZZ} and \cite[Remark 3.3]{HZZ}$)$.
\end{rem}
We will use the following terminology in the sequel.

\begin{definition}{\rm (\cite[Definitions 2.15, 2.19 and 3.9]{NP}) \ Let $(\mathcal{B}, \mathbb{E}, \mathfrak{s})$ be an extriangulated category.

{\rm (1)} \ A sequence $\xymatrix@C=1cm{A\ar[r]^x&B\ar[r]^{y}&C}$ is a {\it conflation} if it realizes an $\mathbb{E}$-extension $\delta\in\mathbb{E}(C, A)$.
In this case, $x$ is called an {\it inflation},  and $y$ a {\it deflation};  and the pair $(\xymatrix@C=1cm{A\ar[r]^x&B\ar[r]^{y}&C}, \delta)$ is called an $\mathbb{E}$-triangle, which is
denoted by
\begin{center} $\xymatrix{A\ar[r]^x&B\ar[r]^{y}&C\ar@{-->}[r]^{\delta}&}$\end{center}
$($the $\delta$ is also omitted if it is not used in the argument$)$.

{\rm (2)} \ Let $\xymatrix{A\ar[r]^x&B\ar[r]^{y}&C\ar@{-->}[r]^{\delta}&}$ and $\xymatrix{A'\ar[r]^{x'}&B'\ar[r]^{y'}&C'\ar@{-->}[r]^{\delta'}&}$
be $\mathbb{E}$-triangles. If a triple $(a, b, c)$ realizes $(a, c): \delta\rightarrow \delta'$, then one writes it as
 $$\xymatrix@R=0.5cm{A\ar[r]^{x}\ar[d]_{a}&B\ar[r]^{y}\ar[d]_{b}&C\ar[d]_{c}\ar@{-->}[r]^{\delta}&\\
 A'\ar[r]^{x'}&B'\ar[r]^{y'}&C'\ar@{-->}[r]^{\delta'}&}$$
 and call $(a, b, c)$ a {\it morphism} of $\mathbb{E}$-triangles.

 {\rm (3)} \ For a conflation $\xymatrix@C=0.6cm{A\ar[r]^-f&B\ar[r]^-g&C,}$ one writes ${\rm Cone}(f) = C$, and ${\rm CoCone}(g) = A$.}
\end{definition}

By (ET4) and (ET4)$^{\rm op}$ one has
\begin{lem}\label{lem1}  Let $(\mathcal{B}, \mathbb{E},\mathfrak{s})$ be an extriangulated category. Then the following statements hold.

\emph{(1)} If $f$ and $g$ are inflations, then $gf$ is an inflation, and there is a conflation
 $$\xymatrix@C=0.6cm{{\rm Cone}(f)\ar[r]&{\rm Cone}(gf)\ar[r]&{\rm Cone}(g)}.$$

\emph{(2)} If $f$ and $g$ are deflations, then $gf$ is a deflation, and there is a conflation
$$\xymatrix@C=0.6cm{{\rm CoCone}(f)\ar[r]&{\rm CoCone}(gf)\ar[r]&{\rm CoCone}(g)}.$$
\end{lem}

\begin{definition}{\rm(\cite[Proposition 3.24,  Definition 3.25]{NP})} {\rm  Let $(\mathcal{B}, \mathbb{E}, \mathfrak{s})$ be an extriangulated category.
An object $P\in\mathcal{B}$ is {\it a projective object} if $\mathbb{E}(P, B)=0$ for any object $B\in\mathcal{B}$. Denote the
subcategory of projective objects by $\mathcal{P}\subseteq\mathcal{B}$. One says that $\mathcal{B}$ has {\it enough projectives} if
any object $B\in\mathcal{B}$ admits a deflation $P\longrightarrow B$ for some $P\in\mathcal{P}$.

Dually, one has the notions of {\it injective objects} and {\it having enough injectives.}}
  \end{definition}

 \begin{definition}{\rm (\cite[Section 5.1]{LN}) Assume that $\mathcal{U}, \mathcal{V}$ are subcategories of $\mathcal{B}$.

$(1)$ \ Let ${\rm CoCone}(\mathcal{U}, \mathcal{V})$ denote the subcategory of $\mathcal{B}$ consisting of $M\in \mathcal{B}$ which admits an $\mathbb{E}$-triangles $\xymatrix{M\ar[r]&U\ar[r]&V\ar@{-->}[r]&}$ in $\mathcal{B}$ with $U\in\mathcal{U}$ and $V\in\mathcal{V}$. Put $\Omega^0(\mathcal{B})=\mathcal{B}$, and define $\Omega^i(\mathcal{B})$ for $i>0$ inductively by
$\Omega^i(\mathcal{B})=\Omega(\Omega^{i-1}(\mathcal{B}))={\rm CoCone}(\mathcal{P}, \Omega^{i-1}(\mathcal{B})).$

$(2)$ \ Let ${\rm Cone}(\mathcal{U}, \mathcal{V})$ denote the subcategory of $\mathcal{B}$ consisting of $M\in \mathcal{B}$ which admits an $\mathbb{E}$-triangles $\xymatrix{U\ar[r]&V\ar[r]&M\ar@{-->}[r]&}$ in $\mathcal{B}$ with $U\in\mathcal{U}$ and $V\in\mathcal{V}$. Put $\Sigma^0(\mathcal{B})=\mathcal{B}$, and define $\Sigma^i(\mathcal{B})$ for $i>0$ inductively by
$\Sigma^i(\mathcal{B})=\Sigma(\Sigma^{i-1}(\mathcal{B}))={\rm Cone}(\Sigma^{i-1}(\mathcal{B}), \mathcal{I}).$}
\end{definition}

Let $(\mathcal{B}, \mathbb{E}, \mathfrak{s})$ be an extriangulated category with enough projectives and enough injectives. Then  $\mathbb{E}(X, \Sigma^i(Y))\cong \mathbb{E}(\Omega^i(X), Y)$, and denote it by $\mathbb{E}^i(X, Y)$ for $i\geqslant 0$ (\cite[Section 5.2]{LN}).

\vskip5pt

Assume that $(\mathcal{B}, \mathbb{E}, \mathfrak{s})$ is an extriangulated category. By Yoneda's Lemma, any $\mathbb{E}$-extension $\delta\in \mathbb{E}(C, A)$ induces  natural transformations
\begin{center} $\delta_\sharp: \mathcal{B}(-, C)\longrightarrow \mathbb{E}(-, A)$ and $\delta^\sharp: \mathcal{B}(A, -)\longrightarrow \mathbb{E}(C, -)$\end{center}
where for any $X\in\mathcal{B}$ one has

$$(\delta_\sharp)_X: \mathcal{B}(X, C)\longrightarrow \mathbb{E}(X, A), \ \  f\mapsto f^*\delta$$
and
$$\delta^\sharp_X: \mathcal{B}(A, X)\longrightarrow \mathbb{E}(C, X), \ \ g\mapsto g_*\delta.$$
\begin{lem} {\rm (\cite[Corollary 3.12]{NP})} \  Let $(\mathcal{B}, \mathbb{E}, \mathfrak{s})$ be an extriangulated category, and \begin{center}$\xymatrix{A\ar[r]^{x}&B\ar[r]^{y}&C\ar@{-->}[r]^{\delta}&}$\end{center}
an $\mathbb{E}$-triangle. Then the following  sequences are exact$:$
$$\xymatrix@C=1cm{\mathcal{B}(C, -)\ar[r]^{\mathcal{B}(y, -)}&\mathcal{B}(B, -)\ar[r]^{\mathcal{B}(x, -)}&\mathcal{B}(A, -)\ar[r]^{\delta^\sharp}&\mathbb{E}(C, -)\ar[r]^{\mathbb{E}(y, -)}&\mathbb{E}(B, -)\ar[r]^{\mathbb{E}(x, -)}&\mathbb{E}(A, -);}$$
$$\xymatrix@C=1cm{\mathcal{B}(-, A)\ar[r]^{\mathcal{B}(-, x)}&\mathcal{B}(-, B)\ar[r]^{\mathcal{B}(-, y)}&\mathcal{B}(-, C)\ar[r]^{\delta_\sharp}&\mathbb{E}(-, A)\ar[r]^{\mathbb{E}(-, x)}&\mathbb{E}(-, B)\ar[r]^{\mathbb{E}(-, y)}&\mathbb{E}(-, C).}$$
\end{lem}

\begin{lem}\label{td=1} \ Let $(\mathcal{B}, \mathbb{E},\mathfrak{s})$ be an extriangulated category. If $f:A\rightarrow B$ and $g:B\rightarrow A$ are morphisms in $\mathcal{B}$ such that $gf=1_{A}$, then $f$ is an inflation if and only if $g$ is a deflation.
\end{lem}
\begin{proof} We justify the ``only if" part. The ``if" part can be dually proved. Assume that $f$ is an inflation with $gf=1_{A}$.
Then there is a morphism of $\mathbb{E}$-triangles
$$\xymatrix@R=0.5cm{A\ar[r]^f\ar@{=}[d]&B\ar[r]^h\ar[d]^g&H\ar[d]&\\
  A\ar[r]^{1_{A}}&A\ar[r]^0&0.&
  }$$
This gives an $\mathbb{E}$-triangle $\xymatrix@C=1,2cm{B\ar[r]^{\tiny\begin{bmatrix}g\\h\end{bmatrix}\ \ \ }&A\oplus H\ar[r]^{\quad0}&0}$ (see \cite[Lemma 3.6(2)]{HZZ}). Thus ${\tiny\begin{bmatrix}g\\h\end{bmatrix} }:B\rightarrow{A\oplus H}$ is an isomorphism, and hence
 a deflation. Therefore  $g=(1,0){\tiny\begin{bmatrix}g\\h\end{bmatrix}}$ is a deflation.
\end{proof}

In the proof of Lemmas \ref{lem2} and \ref{lem3}, we will use the notion of weak pull-back and weak push-out: they are just obtained from the definition of pull-back and push-out, respectively,
by dropping the uniqueness in the following way.

\begin{definition} {\rm Let $\mathcal{M}$ be a category, with commutative square of morphisms$:$
$$\xymatrix@R=0.5cm{Y'\ar[r]^{v'}\ar[d]_{\alpha'}&Z'\ar[d]^{\alpha}\\Y\ar[r]^v&Z.}$$

(1) \ The commutative square is a {\it weak pull-back} of $v$ and $\alpha$, if for any morphism $f: W\rightarrow Y$ and $g: W\rightarrow Z'$ with $vf=\alpha g$, there is a morphism $h: W\rightarrow Y'$ such that $\alpha' h=f$ and $v' h=g$.

(2) \ The commutative square is a {\it weak push-out} of $v'$ and $\alpha'$, if for any morphism $f: Y\rightarrow W$ and $g: Z'\rightarrow W$ with $f\alpha'=g v'$, there is a morphism $h: Y'\rightarrow W$ such that $hv=f$ and $h\alpha=g$.}
\end{definition}

\begin{lem}\label{lem2}  Let $(\mathcal{B}, \mathbb{E},\mathfrak{s})$ be an extriangulated category, with morphisms $f: A\rightarrow B$ and $g: B\rightarrow C$.

\emph{(1)} \ If $f$ is an inflation and  $gf$ is a deflation, then $g$ is a deflation, and there is a conflation
 $$\xymatrix@C=0.6cm{{\rm CoCone}(gf)\ar[r]&{\rm CoCone}(g)\ar[r]&{\rm Cone}(f)}.$$

\emph{(2)} \ If $g$ is a deflation and $gf$ is an inflation, then $f$ is an inflation, and there is a conflation
$$\xymatrix@C=0.6cm{{\rm CoCone}(g)\ar[r]&{\rm Cone}(f)\ar[r]&{\rm Cone}(gf)}.$$
\end{lem}

\begin{proof}  We only prove (1). The assertion (2) can be proved  dually. Assume that there are $\mathbb{E}$-triangles $\xymatrix@=0.8cm{A\ar[r]^f&B\ar[r]^-{f'}&{\rm Cone}(f)\ar@{-->}[r]^-{\delta_f}&}$ and $\xymatrix{{\rm CoCone}(gf)\ar[r]^-{h}&A\ar[r]^{gf}&C\ar@{-->}[r]^{\delta_h}&}$. Thanks to (ET4), there exist morphisms of $\mathbb{E}$-triangles:

$$\xymatrix@R=0.5cm{{\rm CoCone}(gf)\ar[r]^-h\ar@{=}[d]&A\ar[r]^-{gf}\ar[d]_f&C\ar[d]^d\ar@{-->}[r]^{\delta_h}&\\
{\rm CoCone}(gf)\ar[r]^-k&B\ar[r]^{k'}\ar[d]_{f'}&E\ar[d]^e\ar@{-->}[r]^{\delta_k}&\\
&{\rm Cone}(f)\ar@{-->}[d]_{\delta_f}\ar@{=}[r]&{\rm Cone}(f).\ar@{-->}[d]^{\delta_d}&\\
&&&}$$
It follows from \cite[Lemma 3.13]{NP} that the upper right square of above diagram is a weak push-out, thus there is a morphism $t: E\rightarrow C$ such that the diagram commutes:
$$\xymatrix@R=0.5cm{A\ar[r]^f\ar[d]_{gf}&B\ar[d]_{k'}\ar@/^/[ddr]^g&\\
C\ar[r]^{d}\ar@/_/[drr]_{1_C}&E\ar@{-->}[rd]^t&\\
&&C.}$$
Since $td=1_C$ and $d$ is an inflation, it follows from Lemma \ref{td=1} that $t$ is a deflation. Since $k'$ is a deflation, $g=tk'$ is a deflation.

Assume that $\xymatrix{{\rm CoCone}(g)\ar[r]^-{g'}&B\ar[r]^-{g}&C\ar@{-->}[r]^-{\delta_{g'}}&}$ is an $\mathbb{E}$-triangle.
Thanks to \cite[Proposition 3.17]{NP}, there exist morphisms of $\mathbb{E}$-triangles:

$$\xymatrix@R=0.5cm{{\rm CoCone}(gf)\ar[r]^-h\ar[d]_x&A\ar[r]^-{gf}\ar[d]_f&C\ar@{=}[d]\ar@{-->}[r]^-{\delta_h}&\\
{\rm CoCone}(g)\ar[r]^-{g'}\ar[d]_{x'}&B\ar[r]^-{g}\ar[d]_{f'}&C\ar@{-->}[r]^{\delta_{g'}}&\\
{\rm Cone}(f)\ar@{=}[r]\ar@{-->}[d]_{\delta_x}&{\rm Cone}(f)\ar@{-->}[d]^{\delta_f}.&&\\
&&&}$$ So we have a conflation  $\xymatrix@C=0.6cm{{\rm CoCone}(gf)\ar[r]&{\rm CoCone}(g)\ar[r]&{\rm Cone}(f)}$ in $\mathcal{B}$.
\end{proof}

The Extension-Lifting Lemma, stated below, will play an important role in the sequel. It was proved for abelian categories in \cite[VIII, 3.1]{Beli-Reiten} and for exact categories in \cite[Lemma 2.7]{CLZ}.

\begin{lem} {\bf (The Extension-Lifting Lemma)}\label{lem3} \  Let $(\mathcal{B}, \mathbb{E},\mathfrak{s})$ be an extriangulated category and $X, Y\in\mathcal{B}$. Then $\mathbb{E}(X, Y)=0$ if and only if for any commutative diagram with $\mathbb{E}$-triangles

$$\xymatrix@R=0.5cm{&A\ar[r]^i\ar[d]_f&B\ar[r]^d\ar[d]^g&X\ar@{-->}[r]^{\eta}&\\
Y\ar[r]^c&C\ar[r]^p&D\ar@{-->}[r]^\delta&&}$$

\vskip5pt

\noindent there exists a morphism $h:B\rightarrow C$ such that $f=hi$ and $g=ph$.
\end{lem}
\begin{proof} ``$\Leftarrow$". Let $\delta\in\mathbb{E}(X, Y)$. Then there exists a conflation $\xymatrix{Y\ar[r]^c&C\ar[r]^p&X}$ which realizes $\delta$, so we have a commutative diagram with $\mathbb{E}$-triangles
$$\xymatrix@R=0.5cm{&0\ar[r]\ar[d]&X\ar[r]^{1_X}\ar@{=}[d]&X\ar@{-->}[r]^{0}&\\
Y\ar[r]^c&C\ar[r]^p&X\ar@{-->}[r]^\delta&&}$$
Thus there is a morphism $h: X\rightarrow C$ such that $ph=1_X$. Then $\delta=0$. So $\mathbb{E}(X, Y)=0$.

``$\Rightarrow$". Assume $\mathbb{E}(X, Y)=0$. For any commutative diagram with $\mathbb{E}$-triangles
$$\xymatrix@R=0.5cm{&A\ar[r]^i\ar[d]_f&B\ar[r]^d\ar[d]^g&X\ar@{-->}[r]^{\eta}&\\
Y\ar[r]^c&C\ar[r]^p&D\ar@{-->}[r]^\delta&&}$$
there exists a commutative diagram of $\mathbb{E}$-triangles
$$\xymatrix@R=0.5cm{Y\ar[r]^x\ar@{=}[d]&K\ar[d]^k\ar[r]^y&B\ar[d]^g\ar@{-->}[r]^{g^*\delta}&\\
Y\ar[r]^c&C\ar[r]^p&D\ar@{-->}[r]^\delta&}$$
such that the right square is a weak pull-back by the dual of \cite[Corollary 3.16]{NP}.
Since $pf=gi$, there exists a morphism $t:A\rightarrow K$ such that $f=kt$ and $i=yt$.  Note that $i$ is an inflation and $y$ is a deflation. By Lemma \ref{lem2}(2),  $t$ is an inflation.
Assume that $\xymatrix{A\ar[r]^t&K\ar[r]^{t'}&L\ar@{-->}[r]^{\theta}&}$ is an $\mathbb{E}$-triangle. Since $i=yt$, there exists a commutative diagram with $\mathbb{E}$-triangles
$$\xymatrix@R=0.5cm{&Y\ar@{=}[r]\ar[d]_x&Y\ar[d]^{x'}&\\
A\ar[r]^t\ar@{=}[d]&K\ar[r]^{t'}\ar[d]_y&L\ar[d]^s\ar@{-->}[r]^{\theta}&\\
A\ar[r]^i&B\ar[r]^{d}\ar@{-->}[d]_{g^*\delta}&X\ar@{-->}[d]^{\rho}\ar@{-->}[r]^{\eta}&\\
&&&}$$
with $\theta=s^*\eta$ and $d^*\rho=g^*\delta$, by the dual of \cite[Proposition 3.17]{NP}. Hence $\rho=0$ as $\mathbb{E}(X, Y)=0$, which implies $g^*\delta=0$. So there exists a morphism $r: B\rightarrow K$ such that $yr=1_B$. Then
$p(f-kri)= gi - gyri = gi - gi = 0$. Thus there exists a morphism $m:A\rightarrow Y$ such that $f-kri=cm$. By exactness of the sequence $\xymatrix{\mathcal{B}(B, Y)\ar[r]^{\mathcal{B}(i, Y)}&\mathcal{B}(A, Y)\ar[r]&\mathbb{E}(X, Y)=0,}$ there is a morphism $n: B\rightarrow Y$ such that $m=ni$. Put $h=cn+kr: B \to C$. Then $f=hi$ and $ph=g$, as desired.
\end{proof}

\begin{definition}{\rm(\cite[Condition 5.8]{NP}) An extriangulated category $(\mathcal{B}, \mathbb{E},\mathfrak{s})$ satisfies {\it Condition} (WIC) if the following two conditions hold:

(1) \ For $f\in \mathcal{B}(A, B)$ and $g\in\mathcal{B}(B, C)$,  if  $gf$ is an inflation, then so is $f$.

(2) \ For $f\in \mathcal{B}(A, B)$ and $g\in\mathcal{B}(B, C)$, if $gf$ is a deflation, then so is $g$.}
\end{definition}

\begin{remark} {\rm (1)} \ If  $\mathcal{B}$ is an exact category, then $(\mathcal{B}, \mathbb{E},\mathfrak{s})$ satisfies  Condition {\rm (WIC)} if and only if $\mathcal{B}$ is weakly idempotent complete.

{\rm (2)} If $\mathcal{B}$ is a triangulated category, then Condition {\rm (WIC)} is automatically satisfied.
\end{remark}

It is worth noting that an extriangulated category $\mathcal{B}$ satisfies Condition (WIC) if and only if $\mathcal{B}$ is weakly idempotent complete, see \cite[Proposition C]{K}.

\subsection{Cotorsion pairs in extriangulated categories}

Let $(\mathcal{B}, \mathbb{E},\mathfrak{s})$ be an extriangulated category. For any class $\mathcal{X}$ and $\mathcal{Y}$ of objects of $\mathcal{B}$, we write $\mathbb{E}(\mathcal{X}, \mathcal{Y})=0$ provided that $\mathbb{E}(X, Y)=0$ for all $X\in\mathcal{X}$ and $Y\in\mathcal{Y}$. Put
$$\mathcal{X}^\perp=\{Z\in\mathcal{B}\mid \mathbb{E}(X, Z)=0, \ \forall \ X\in\mathcal{X}\},  \mbox{and} \ \ ^\perp\mathcal{Y}=\{Z\in\mathcal{B}\mid \mathbb{E}(Z, Y)=0, \ \forall \ Y\in\mathcal{Y}\}.$$

\begin{definition} {\rm Let $(\mathcal{B}, \mathbb{E},\mathfrak{s})$ be an extriangulated category and $\mathcal{Z}$ a class of objects in $\mathcal{B}$.

(1) \ \  $\mathcal{Z}$ is {\it closed under cocones of deflations},
if for any $\mathbb{E}$-triangle $\xymatrix{A\ar[r]^x&B\ar[r]^y&C\ar@{-->}[r]^\delta&}$ with $B\in\mathcal{Z}$ and $C\in\mathcal{Z}$,  one has $A\in\mathcal{Z}$.

(2) \ \  $\mathcal{Z}$ is {\it closed under cones of inflations}, if for any  \ $\mathbb{E}$-triangle $\xymatrix{A\ar[r]^x&B\ar[r]^y&C\ar@{-->}[r]^\delta&}$ with $A\in\mathcal{Z}$ and $B\in\mathcal{Z}$,  one has $C\in\mathcal{Z}$.

(3) \ \ $\mathcal{Z}$ is {\it closed under extensions}, if for any $\mathbb{E}$-triangle $\xymatrix{A\ar[r]^x&B\ar[r]^y&C\ar@{-->}[r]^\delta&}$ with $A\in\mathcal{Z}$ and $C\in\mathcal{Z}$,  one has $B\in\mathcal{Z}$.}
\end{definition}

\begin{definition} {\rm Assume that $(\mathcal{B}, \mathbb{E},\mathfrak{s})$ is an extriangulated category. Let $\mathcal{X},\mathcal{Y}\subseteq\mathcal{B}$ be a pair of full additive subcategories which are closed under isomorphisms and direct summands.

 (1) \ \ The pair $(\mathcal{X, Y})$ is  a {\it cotorsion pair}, provided that $\mathcal{X}^\perp=\mathcal{Y}$ and $^\perp\mathcal{Y}=\mathcal{X}$.

 (2) \ \ A  cotorsion pair $(\mathcal{X}, \mathcal{Y})$ is  {\it complete} if it satisfies the following conditions:

 For any $C\in\mathcal{B}$, there are $\mathbb{E}$-triangles
 $$\xymatrix{Y_C\ar[r]&X_C\ar[r]&C\ar@{-->}[r]&}~ \textrm{and} ~\xymatrix{C\ar[r]&Y^C\ar[r]&X^C\ar@{-->}[r]&}$$
\noindent with $X_C\in\mathcal{X}, \ X^C\in\mathcal{X}$,  $Y_C\in\mathcal{Y}$, and $Y^C\in\mathcal{Y}$.

(3) \ \ A cotorsion pair $(\mathcal{X}, \mathcal{Y})$ is  \emph{hereditary} if $\mathcal{X}$ is closed under cocones of deflations and $\mathcal{Y}$ is closed under cones of inflations.
\vspace{1mm}}
\end{definition}

\begin{rem}\label{fact2.18} Assume that $\mathcal{X}$ and $\mathcal{Y}$ are full additive subcategories of an extriangulated category $\mathcal{B}$ which are closed under isomorphisms and direct summands.
It is well-known that the pair $(\mathcal{X}, \mathcal{Y})$ is a complete cotorsion pair in $\mathcal{B}$ if and only if $\mathbb{E}(\mathcal{X}, \mathcal{Y})=0$ and for any $C\in\mathcal{B}$ there are $\mathbb{E}$-triangles
$$\xymatrix{Y_C\ar[r]&X_C\ar[r]&C\ar@{-->}[r]&}~\textrm{and}~\xymatrix{C\ar[r]&Y^C\ar[r]&X^C\ar@{-->}[r]&}$$

\noindent with  $X_C, \ X^C\in\mathcal{X}$ and  $Y_C, \ Y^C\in\mathcal{Y}$. See e.g. \cite[Remark 4.4]{NP}.
\vspace{1mm}
\end{rem}

\begin{prop}\label{pro0} Let $(\mathcal{B}, \mathbb{E},\mathfrak{s})$ be an extriangulated category and $(\mathcal{X}, \mathcal{Y})$ a complete cotorsion pair. Then the following  conditions are equivalent:

{\rm (1)} \  $\mathcal{Y}$ is closed under cones of inflations.

{\rm (2)} \  $\mathcal{X}$ is closed under cocones of deflations.

{\rm (3)} \   $(\mathcal{X}, \mathcal{Y})$ is hereditary.

Moreover, if $\mathcal{B}$ has enough projectives and injectives, then the above conditions are also equivalent to$:$

{\rm (4)} \   $\mathbb{E}^{i}(X, Y)=0$ for each $i\geqslant2$ and all $X\in{\mathcal{X}}$ and $Y\in{\mathcal{Y}}$.

\end{prop}

\begin{proof}

$(1)\Rightarrow(2)$. Assume that $\xymatrix{A\ar[r]^x&X_1\ar[r]^y&X_2\ar@{-->}[r]&}$ is an $\mathbb{E}$-triangle with $X_1, X_2\in\mathcal{X}$. We claim that any morphism $f: A\rightarrow Y$ with $Y\in\mathcal{Y}$ factors through $\omega=\mathcal{X}\cap\mathcal{Y}$.
Since $\mathbb{E}(X_2, Y)=0$, there is a morphism $g: X_1\rightarrow Y$ such that $f=gx$. As $(\mathcal{X}, \mathcal{Y})$ is a complete cotorsion pair, there is an $\mathbb{E}$-triangle $\xymatrix{X_1\ar[r]&Y'\ar[r]&X\ar@{-->}[r]&}$ with $X\in\mathcal{X}$, and $Y'\in\omega$ because $\mathcal{X}$ is closed under extensions. So $g$ can factor through $Y'\in \omega$ as $\mathbb{E}(X, Y)=0$.

Since $(\mathcal{X}, \mathcal{Y})$ is a complete cotorsion pair, there are $\mathbb{E}$-triangles $$\xymatrix{Y_1\ar[r]&X_3\ar[r]&A\ar@{-->}[r]&}~\mbox{and}~\xymatrix{X_3\ar[r]&W\ar[r]&X_4\ar@{-->}[r]&}$$
with $X_3, X_4\in\mathcal{X}$, $Y_1\in\mathcal{Y}$, and $W\in\omega$ because $\mathcal{X}$ is closed under extensions. By (ET4) there is a commutative diagram with $\mathbb{E}$-triangles

$$\xymatrix@R=0.5cm{Y_1\ar[r]^h\ar@{=}[d]&X_3\ar[r]\ar[d]&A\ar[d]^f\ar@{-->}[r]&\\
Y_1\ar[r]&W\ar[r]\ar[d]&Y''\ar[d]\ar@{-->}[r]&\\
&X_4\ar@{-->}[d]\ar@{=}[r]&X_4\ar@{-->}[d]&\\
&&&}$$
Note that  $Y''\in\mathcal{Y}$, since $\mathcal{Y}$ is closed under cones of inflations. Then $f$ factors through $\omega$, and hence
$f$ factors through $W$. Therefore, $1_{Y_1}$ factors through $h$ by \cite[Corollary 3.5]{NP}, which implies that $A$ is a direct summand of $X_3$. So  $A\in\mathcal{X}$.

$(2)\Rightarrow(1)$. The proof is dual to that of $(1)\Rightarrow(2)$.

$(3)\Rightarrow(1)$ is obvious.

$(2)\Rightarrow(3)$ is clear because of the equivalence of (1) and (2).

If $\mathcal{B}$ has enough projectives and enough injectives, then the equivalence of $(1)$ and $(4)$ follows from  \cite[Lemma 4.3]{LZ}.
\end{proof}


\subsection{Model structures}

Good and standard references are \cite{Q1}, \cite{Q2}, \cite{Hovey} and \cite{Hir}.

\begin{definition} {\rm (\cite{Q1}; \cite[p.233]{Q2}; \cite[Definition 1.2.3]{Hovey})} {\rm  A {\it  model structure} on a category $\mathcal{M}$ is a triple $({\rm CoFib, Fib, Weq})$ of classes of morphisms, where the morphisms in the three classes are respectively called {\it cofibrations, fibrations}, and {\it weak equivalences}, satisfying the following axioms:

{\bf (Two out of three axiom)} Let  $f: A\rightarrow B$ and $g: B\rightarrow C$ be a pair of composable morphisms in $\mathcal{M}$. If two of the morphisms $f, g, gf$ are weak equivalences, then so is the third one.

{\bf (Retract axiom)} If $g$ is a retract of $f$, and $f$ is a cofibration (a fibration, a weak equivalence, respectively), then so is $g$.

{\bf (Lifting axiom)} Cofibrations have the left lifting property with respect to all morphisms in ${\rm Fib}\cap{\rm Weq}$, and fibrations have the right lifting property with respect to all morphisms in ${\rm CoFib}\cap{\rm Weq}$. That is, given a commutative diagram

$$\xymatrix@R=0.5cm{A\ar[r]^a\ar[d]_i&X\ar[d]^p\\
B\ar[r]_b\ar@{-->}[ur]^h&Y}$$

\noindent If either $i\in{\rm CoFib}$ and $p\in{\rm Fib}\cap{\rm Weq}$,  or $i\in{\rm CoFib}\cap{\rm Weq}$ and $p\in{\rm Fib}$, then there exists a morphism $h: B\rightarrow X$ such that $a=hi, b=ph$.

{\bf (Factorization axiom)} Any morphism $f: A\rightarrow B$ admits factorizations $f=pi$ and $f=qj$, where $i\in{\rm CoFib}\cap{\rm Weq}, p\in{\rm Fib}, j\in{\rm CoFib}$ and $q\in{\rm Fib}\cap{\rm Weq}$.
\vspace{1mm}

A \emph{model category} is a category $\mathcal{M}$ with finite limits and colimits together with a model structure. }
\end{definition}
The morphisms in  ${\rm CoFib}\cap{\rm Weq}$ (respectively,  ${\rm Fib}\cap{\rm Weq}$) are called {\it trivial cofibrations} (respectively, {\it trivial fibrations}). Put ${\rm TCoFib}:={\rm CoFib}\cap{\rm Weq}$ and  ${\rm TFib}:={\rm Fib}\cap{\rm Weq}$.

%


\begin{fact}\label{fact2.21}{\rm (\cite{Q1})} Let $({\rm CoFib}, {\rm Fib}, {\rm Weq})$ be a model structure on a category $\mathcal{M}$ with zero object. Then

$(1)$ Both the classes ${\rm CoFib}$ and ${\rm Fib}$ are closed under compositions.

$(2)$ Isomorphisms are fibrations, cofibrations, and weak equivalences.

$(3)$ Cofibrations are closed under push-outs, i.e., given a push-out square

\begin{center}$\xymatrix@R=0.5cm{\bullet\ar[d]_i\ar[r]&\bullet\ar@{-->}[d]^{i'}\\\bullet\ar@{-->}[r]&\bullet}$
\end{center}
with $i\in{\rm CoFib}$, then $i'\in{\rm CoFib}$.

Also, trivial cofibrations are closed under push-outs.

$(4)$ Fibrations are closed under pull-backs, and trivial fibrations are closed under pull-backs.
\end{fact}

A striking property of a model structure is that any two classes of ${\rm CoFib, Fib, Weq}$ uniquely determines the third.

\begin{prop}\label{prop2.22} $($\cite[p.234]{Q2}$)$ Let $({\rm CoFib}, {\rm Fib}, {\rm Weq})$ be a model structure on a category $\mathcal{M}$. Then

$(1)$ Cofibrations are precisely those morphisms which have the left lifting property with respect to all the trivial fibrations.

 $(2)$ Trivial cofibrations are precisely those morphisms which have the left lifting property with respect to all the fibrations.

 $(3)$ Fibrations are precisely those morphisms which have the right lifting property with respect to all the trivial cofibrations.

  $(4)$ Trivial fibrations are precisely those morphisms which have the right lifting property with respect to all the  cofibrations.

 $(5)$ ${\rm Weq}={\rm TFib}\circ{\rm TCoFib}$.
\end{prop}

For a model structure $({\rm CoFib, Fib, Weq})$ on a category $\mathcal{M}$ with zero object, an object $X$ is {\it trivial} if $0\rightarrow X$ is a weak equivalence, or equivalently, $X\rightarrow 0$ is a weak equivalence. It is {\it cofibrant} if $0\rightarrow X$ is a cofibration, and it is  {\it fibrant} if $X\rightarrow 0$ is a fibration. An object is \emph{trivially cofibrant} (respectively, \emph{trivially fibrant}) if it is both trivial and cofibrant (respectively, fibrant).

\subsection{The homotopy category of a model structure}

Quillen's homotopy category of a model structure $($CoFib, Fib, Weq$)$ on category $\mathcal{M}$ is by definition the localization $\mathcal{M}[{\rm Weq}^{-1}]$
of $\mathcal{M}$ with respect to {\rm Weq}, which is denoted by ${\rm Ho}(\mathcal{M})$. This is an important object of study in algebra and topology. Let $\mathcal{M}$ be a model category. Then $\mathcal{M}$ has the initial object and the final object (and then
there is the notion of cofibrant objects and fibrant objects), and $\mathcal{M}$ has pull-backs and push-outs. Let $\mathcal{M}_{cf}$ be the full subcategory of  both fibrant and cofibrant objects of $\mathcal{M}$.  It is proved
that the left homotopy relation $\overset{l}{\sim}$ coincides with the right homotopy relation $\overset{r}{\sim}$ on $\mathcal{M}_{cf}$, which
is denoted by $\sim$. Then $\sim$ is an equivalence relation ideal of $\mathcal{M}_{cf}$ (\cite{Q1}, Lemmas 4, 5, and
their duals). The corresponding quotient category is denoted by $\pi\mathcal{M}_{cf}$, and the composition of
the embedding $\mathcal{M}_{cf}\hookrightarrow \mathcal{M}$  and the localization functor $\mathcal{M}\rightarrow {\rm Ho}(\mathcal{M})$ induces an equivalence
$\pi\mathcal{M}_{cf}\cong{\rm Ho}(\mathcal{M})$. See \cite[Theorem 1']{Q1}.

This important theorem has been named as the {\it fundamental theorem of model categories}.
See M. Hovey \cite[Theorem 1.2.10]{Hovey}; also J. Gillespie \cite{Gillespie}, and W. G. Dwyer-J. Spalinski \cite{DS}.

Some important categories can not be a model category. For example, the existence of pull-backs
 and push-outs can not be guaranteed in extriangulated categories. Thus, to apply the
fundamental theorem of model categories in extriangulated categories, one has to weaken the
condition that $\mathcal{M}$ is a model category. For this purpose, we consider the following conditions.

\begin{cond}\label{condition} Let $\mathcal{M}$ be a category satisfying the following conditions$:$

{\rm (i)} $\mathcal{M}$ has the initial object and final object$;$

{\rm (ii)} $\mathcal{M}$ has finite coproducts and finite products$;$

{\rm (iii)} There is a model structure $({\rm CoFib}, {\rm Fib}, {\rm Weq})$ on $\mathcal{M};$

{\rm (iv)} For any trivial cofibration $i: A\rightarrow B$ and any morphism $u: A\rightarrow C$, there exists a
weak push-out square

\begin{center}$\xymatrix@R=0.5cm{A\ar[r]^u\ar[d]_i&C\ar[d]^{i'}\\B\ar[r]&D}$\end{center}

\noindent such that $i'$ is also a trivial cofibration$;$

{\rm (v)} For any trivial fibration $p: C\rightarrow D$ and any morphism $u: B\rightarrow D$, there exists a
weak pull-back square

\begin{center}$\xymatrix@R=0.5cm{A\ar[r]\ar[d]_{p'}&C\ar[d]^{p}\\B\ar[r]^u&D}$\end{center}

\noindent such that $p'$ is also a trivial fibration.

\end{cond}

If $\mathcal{M}$ is a model category, then all the five conditions in Condition \ref{condition} are satisfied.

\vskip5pt

 By carefully checking  the proof of Quillen \cite[Theorem 1']{Q1} and Hovey \cite[Theorem 1.2.10]{Hovey}, one has the following.

\begin{thm}{\rm(Fundamental Theorem of Model Structures, \cite[Theorem 1']{Q1}, \cite[Theorem 1.2.10]{Hovey})}\label{thm}
Assume that $\mathcal M$ is a category
satisfying the conditions ${\rm (i)}$, ${\rm (ii)}$,  ${\rm (iii)}$, ${\rm (iv)}$, ${\rm (v)}$ in {\rm Condition \ref{condition}}. Then
the composition of
the embedding $\mathcal{M}_{cf}\hookrightarrow \mathcal{M}$  and the localization functor $\mathcal{M}\rightarrow {\rm Ho}(\mathcal{M})$ induces an equivalence
$\pi\mathcal{M}_{cf}\cong{\rm Ho}(\mathcal{M})$ as categories.
\end{thm}

 \begin{remark} We are grateful to the referee for pointing us to \cite[Theorem 3.2]{Egger}, which also gets Theorem \ref{thm} under the assumption that $\mathcal{M}$ is a pointed category with finite products and finite coproducts equipped with a model structure.
\end{remark}
%
%


\section{\bf Proof of Theorem \ref{thmA}}

Throughout $\mathcal{B}=(\mathcal{B}, \mathbb{E},\mathfrak{s})$ is an extriangulated category, and if the weakly idempotent completeness of $\mathcal{B}$ is needed, we will explicitly specify this condition.
Let $\mathcal{X}$ and $\mathcal{Y}$ be full additive subcategories of $\mathcal{B}$
which are closed under direct summands and isomorphisms. Put $\omega=\mathcal{X}\cap \mathcal{Y}$. Define five classes of morphisms in $\mathcal{B}$ as follows.

\begin{center}${\rm CoFib}_{\omega}=\{f: A\rightarrow B\mid f $ is an inflation with ${\rm Cone}(f)\in\mathcal{X}\}$.
\vspace{2mm}

${\rm Fib}_{\omega}=\{f: A\rightarrow B\mid f$ is $\omega$-epic, i.e., $\mathcal{B}(W, f)$ is surjective for any $W\in\omega\}$.
\vspace{2mm}

 ${\rm TCoFib}_{\omega}=\{f: A\rightarrow B\mid f$ is a splitting inflation with ${\rm Cone}(f)\in\omega\}$.  \qquad (3.1)
\vspace{2mm}

${\rm TFib}_{\omega}=\{f:A\rightarrow B\mid f$ is a deflation with ${\rm CoCone}(f)\in\mathcal{Y}\}$.
\vspace{2mm}

${\rm Weq}_{\omega}=\{f: A\rightarrow B\mid $ there is a deflation $(f, t): A\oplus W\rightarrow B$

 such that $W\in\omega$ and ${\rm CoCone}(f, t)\in\mathcal{Y}\}$. \end{center}
\vspace{2mm}

 Therefore, a morphism $f: A\rightarrow B$ is in ${\rm Weq}_\omega$ if and only if there is a commutative diagram
$$\xymatrix@R=0.5cm{A\ar[rr]^f\ar[dr]_{\tiny\begin{pmatrix}1\\0\end{pmatrix}}&&B\\&A\oplus W\ar[ur]_{(f,~ t)}&}$$
\noindent such that $W\in\omega, (f, t)$ is a deflation, and ${\rm CoCone}(f, t)\in\mathcal{Y}$. Since ${\tiny\begin{pmatrix}1\\0\end{pmatrix}}\in {\rm TCoFib}_\omega$, ${\rm Weq}_\omega$ can be reformulated as
$${\rm Weq}_\omega=\{gf|g\in{\rm TFib}_\omega, f\in{\rm TCoFib}_\omega\}.$$

\subsection{\bf Model structure induced by a hereditary complete cotorsion pair}

The aim of this subsection is to prove the ``if" part of Theorem \ref{thmA}.
Thus, we must check that the triple $({\rm CoFib}_{\omega}, {\rm Fib}_{\omega}, {\rm Weq}_{\omega})$ in Theorem \ref{thmA} satisfy the factorization, two-out-of-three, retract and lifting axioms for a model structure, under the condition that
$\mathcal{B}=(\mathcal{B}, \mathbb{E}, \mathfrak{s})$ is a weakly idempotent complete extriangulated category, $(\mathcal{X}, \mathcal{Y})$ is a hereditary complete cotorsion pair, such that $\omega=\mathcal{X}\cap \mathcal{Y}$ is contravariantly finite in
$\mathcal{B}$. However, not every lemma needs all these conditions, and we will explicitly specify what we exactly need.

\vskip5pt

First, we give a description  ${\rm TCoFib}_\omega$ and ${\rm TFib}_\omega$. Since $1_X\in {\rm TFib}_\omega$ and $1_X\in {\rm TCoFib}_\omega$ for any object $X\in\mathcal{B}$, we have ${\rm TCoFib}_\omega\subseteq {\rm Weq}_\omega$ and ${\rm TFib}_\omega\subseteq {\rm Weq}_\omega$.

\begin{prop}\label{pro1} \ Assume that $\mathcal{B}=(\mathcal{B}, \mathbb{E}, \mathfrak{s})$ is a weakly idempotent complete extriangulated category. Let $\mathcal{X}, \mathcal{Y}$ be full additive subcategories of $\mathcal{B}$ which are closed under direct summands and isomorphisms, and $\omega=\mathcal{X}\cap \mathcal{Y}$. If $\mathbb{E}(\mathcal{X}, \mathcal{Y})=0$, then we have
$${\rm TCoFib}_\omega={\rm CoFib}_\omega\cap{\rm Weq}_\omega,\quad{\rm TFib}_\omega={\rm Fib}_\omega\cap{\rm Weq}_\omega.$$
\end{prop}
\begin{proof} We first prove ${\rm TCoFib}_\omega={\rm CoFib}_\omega\cap{\rm Weq}_\omega$. It is easy to see that ${\rm TCoFib}_\omega\subseteq {\rm CoFib}_\omega\cap{\rm Weq}_\omega$.
Conversely, if $f: A\rightarrow B$ is in ${\rm CoFib}_\omega\cap{\rm Weq}_\omega$, then $f$ is an inflation with ${\rm Cone}(f)\in\mathcal{X}$ and there is an $\mathbb{E}$-triangle $\xymatrix{Y\ar[r]^-y&A\oplus W\ar[r]^-{(f, t)}&B\ar@{-->}[r]^{\delta}&}$ with $W\in\omega$ and $Y\in\mathcal{Y}$.
Suppose that $\xymatrix{A\ar[r]^f&B\ar[r]^-{g}&{\rm Cone}(f)\ar@{-->}[r]^-{\eta}&}$ is an $\mathbb{E}$-triangle with ${\rm Cone}(f)\in\mathcal{X}$. Since $f=(f, t){\tiny\begin{pmatrix}1\\0\end{pmatrix}}$, there exists a commutative diagram of $\mathbb{E}$-triangles
$$\xymatrix@R=0.7cm{&Y\ar@{=}[r]\ar[d]_y&Y\ar[d]^{t}&\\
A\ar[r]^{\binom{1}{0}\ \ }\ar@{=}[d]&A\oplus W\ar[r]^{(0, 1)}\ar[d]_{(f,t)}&W\ar[d]^s\ar@{-->}[r]^0&\\
A\ar[r]^f&B\ar[r]^-{g}\ar@{-->}[d]_{\delta}&{\rm Cone}(f)\ar@{-->}[d]^{\rho}\ar@{-->}[r]^-{\eta}&\\
&&&}$$
with $\delta=g^*\rho$ by the dual of \cite[Proposition 3.17]{NP}. Since $\mathbb{E}(\mathcal{X}, \mathcal{Y})=0$, we have $\rho=0$, which implies that ${\rm Cone}(f)\in\omega$ as ${\rm Cone}(f)$ is  a direct summand of $W$. Then $s$
is a splitting deflation and hence $g$ is a splitting deflation. Thus $f$ is a splitting inflation, and $f\in{\rm TCoFib}_\omega$ by definition. This proves ${\rm TCoFib}_\omega={\rm CoFib}_\omega\cap{\rm Weq}_\omega$.

Next we prove ${\rm TFib}_\omega={\rm Fib}_\omega\cap{\rm Weq}_\omega$. Since $\mathbb{E}(\mathcal{X}, \mathcal{Y})=0$, ${\rm TFib}_\omega\subseteq{\rm Fib}_\omega$, and hence ${\rm TFib}_\omega\subseteq{\rm Fib}_\omega\cap{\rm Weq}_\omega$. Conversely, if $f\in {\rm Fib}_\omega\cap{\rm Weq}_\omega$, then there is an $\mathbb{E}$-triangle $$\xymatrix{{\rm CoCone}(f,t)\ar[r]^{\quad x}&A\oplus W\ar[r]^{\quad(f,~ t)}&B\ar@{-->}[r]^{\delta}&}$$ with $W\in\omega$ and ${\rm CoCone}(f,t)\in\mathcal{Y}.$
Since $f$ is $\omega$-epic, there is a morphism $s:W\rightarrow A$ such that $t=fs$. Then $(f, t)=f(1, s)$.
Since a weakly idempotent complete extriangulated category satisfies Condition (WIC), it follows that $f$ is a deflation. So we have a commutative diagram with $\mathbb{E}$-triangles
$$\xymatrix@R=0.5cm{{\rm CoCone}(f)\ar[r]^-r\ar@{-->}[d]_g&A\ar[r]^-{f}\ar[d]_{\binom{1}{0}}&B\ar@{=}[d]\ar@{-->}[r]^{\eta}&\\
{\rm CoCone}(f,t)\ar[r]^-x\ar@{-->}[d]_h&A\oplus W\ar[r]^-{(f, t)}\ar[d]_{(1, s)}&B\ar@{=}[d]\ar@{-->}[r]^{\delta}&\\
{\rm CoCone}(f)\ar[r]^-r&A\ar[r]^-{f}&B\ar@{-->}[r]^{\eta}&}$$
by (ET3)$^{\rm op}$. Therefore, $hg$ is an isomorphism, which implies that ${\rm CoCone}(f)$ is a direct summand of ${\rm CoCone}(f, t)$, thus ${\rm CoCone}(f)\in\mathcal{Y}$. So $f\in {\rm TFib}_\omega$. This completes the proof.
\end{proof}

\begin{lem}\label{pro2} \ Let $\mathcal{B}=(\mathcal{B}, \mathbb{E},\mathfrak{s})$ be an extriangulated category,
let $\mathcal{X}$ and $\mathcal{Y}$ be full subcategories of $\mathcal{B}$ which are closed under extensions, direct summands, and isomorphisms. Then the classes
${\rm CoFib}_\omega, ~{\rm Fib}_\omega,~ {\rm TCoFib}_\omega$ and ${\rm TFib}_\omega$ are closed under compositions.
\end{lem}
\begin{proof} For ${\rm TCoFib}_\omega$, this follows from  the fact that splitting inflations are closed under compositions, that $\omega$ is extension closed, and Lemma \ref{lem1}. Similarly for the other cases.
\end{proof}

The following is \textbf{the Factorization axiom}. By Proposition \ref{pro1} we already have ${\rm TCoFib}_\omega={\rm CoFib}_\omega\cap{\rm Weq}_\omega$ and ${\rm TFib}_\omega={\rm Fib}_\omega\cap{\rm Weq}_\omega.$

\begin{prop}\label{pro3} \ Assume that $\mathcal{B}=(\mathcal{B}, \mathbb{E}, \mathfrak{s})$ is a weakly idempotent complete extriangulated category. Let $(\mathcal{X}, \mathcal{Y})$ be a complete cotorsion pair,
such that $\omega=\mathcal{X}\cap \mathcal{Y}$ is contravariantly finite in $\mathcal{B}$. Then we have ${\rm Mor}(\mathcal{B})= {\rm TFib}_\omega\circ{\rm CoFib}_\omega={\rm Fib}_\omega\circ {\rm TCoFib}_\omega$.
\end{prop}
\begin{proof} \ We can prove ${\rm Mor}(\mathcal{B})={\rm Fib}_\omega\circ {\rm TCoFib}_\omega$ by the same argument of \cite[The first factorization]{CLZ}. This needs the condition that
$\omega =\mathcal{X}\cap \mathcal{Y}$ is contravariantly finite in $\mathcal{B}$.

Next we prove ${\rm Mor}(\mathcal{B})= {\rm TFib}_\omega\circ{\rm CoFib}_\omega$. For every morphism $f: A\rightarrow B$,
since $(\mathcal{X}, \mathcal{Y})$ is a complete cotorsion pair,
there exists an $\mathbb{E}$-triangle $\xymatrix{Y_B\ar[r]&X_B\ar[r]^{t}&B\ar@{-->}[r]&}$ with $X_B\in\mathcal{X}$ and $Y_B\in\mathcal{Y}$. Hence there is a deflation $(f,~ t): A\oplus X_B\rightarrow B$ by the dual of \cite[Corollary 3.16]{NP}. Assume that $\xymatrix{K\ar[r]&A\oplus X_B\ar[r]^{\quad(f,~ t)}&B\ar@{-->}[r]&}$ is an $\mathbb{E}$-triangle. There exists an  $\mathbb{E}$-triangle $\xymatrix@R=0.6cm{K\ar[r]&Y^K\ar[r]&X^K\ar@{-->}[r]&}$ with $X^K\in\mathcal{X}$ and $Y^K\in\mathcal{Y}$. It follows from \cite[Proposition 3.15]{NP} that there is a commutative diagram of $\mathbb{E}$-triangles as follows.
$$\xymatrix{K\ar[r]\ar[d]&A\oplus X_B\ar[r]^-{(f, t)}\ar[d]_{i}&B\ar@{=}[d]\ar@{-->}[r]&\\
Y^K\ar[r]^{g}\ar[d]&E\ar[r]^{h}\ar[d]&B\ar@{-->}[r]&\\
X^K\ar@{=}[r]\ar@{-->}[d]&X^K\ar@{-->}[d]&&\\
&&&}$$
Then $i\in {\rm CoFib}_\omega$ and $h\in{\rm TFib}_\omega$ by definition.

Since $\binom{1}{0}: A\rightarrow A\oplus X_B\in{\rm CoFib}_\omega,$ one has $i\circ\binom{1}{0}\in{\rm CoFib}_\omega$ by Lemma \ref{pro2}. Thus $f=h(i\circ\binom{1}{0})\in  {\rm TFib}_\omega\circ{\rm CoFib}_\omega$, as desired. \end{proof}

\begin{lem}\label{lem5} \ Assume that $\mathcal{B}=(\mathcal{B}, \mathbb{E}, \mathfrak{s})$ is a weakly idempotent complete extriangulated category. Let $(\mathcal{X}, \mathcal{Y})$ be a hereditary complete cotorsion pair,
such that $\omega=\mathcal{X}\cap \mathcal{Y}$ is contravariantly finite in $\mathcal{B}$.
Then ${\rm TCoFib}_\omega\circ{\rm TFib}_\omega\subseteq{\rm Weq}_\omega.$ Furthermore, the class ${\rm Weq}_\omega$ is closed under compositions.
\end{lem}
\begin{proof} Let $a: A\to A'$ be in ${\rm TFib}_\omega$ and $x': A'\to B'$ in ${\rm TCoFib}_\omega$. Hence there are some $x\in{\rm CoFib}_\omega$ and $b\in{\rm TFib}_\omega$ such that $x'a=bx$ by Proposition \ref{pro3}. It suffices to show $x\in{\rm TCoFib}_\omega$. There are $\mathbb{E}$-triangles $\xymatrix{A'\ar[r]^{x'}&B'\ar[r]^{y'}&W\ar@{-->}[r]^{\delta'}&}$, $\xymatrix{A\ar[r]^{x}&B\ar[r]^{y}&X\ar@{-->}[r]^\delta&}$, $\xymatrix{Y\ar[r]^{k}&A\ar[r]^{a}&A'\ar@{-->}[r]^\eta&}$ and $\xymatrix{Y'\ar[r]^{k'}&B\ar[r]^{b}&B'\ar@{-->}[r]^\theta&}$ with $Y, Y'\in\mathcal{Y}, X\in\mathcal{X}, W\in\omega$ such that $x'$ is a splitting inflation by definition. Therefore there is a commutative diagram with $\mathbb{E}$-triangles
$$\xymatrix@R=0.7cm{
  Y\ar@{}[dr]|{\circlearrowleft} \ar[d]_{k}\ar[r]^m  & Y' \ar@{}[dr]|{\circlearrowleft}\ar[d]_{k'}\ar[r]^{m'}  & Y'' \ar[d]_{k''}\ar@{-->}[r]^{\delta''}  & \\
  A\ar@{}[dr]|{\circlearrowleft} \ar[d]_{a} \ar[r]^{x} & B \ar@{}[dr]|{\circlearrowleft}\ar[d]_{b} \ar[r]^{y} &X \ar[d]_{c} \ar@{-->}[r]^{\delta} &  \\
  A' \ar@{-->}[d]_{\eta} \ar[r]^{x'} & B' \ar@{-->}[d]_{\theta} \ar[r]^{y'} & W \ar@{-->}[d]_{\rho} \ar@{-->}[r]^{\delta'} &  \\
  &  & &  }
  $$
such that $(k, k', k''), (m, x, x')$ and $(m', y, y')$ are morphisms of $\mathbb{E}$-triangles by \cite[Lemma 5.9]{NP}. Note that $Y, Y'\in\mathcal{Y}$ and $\mathcal{Y}$ is closed under cones of inflations. Then we have $Y''\in\mathcal{Y}$. It is easy to see that $X\in\omega$ as $\mathcal{Y}$ is closed under extensions. Since $\mathbb{E}(\mathcal{X}, \mathcal{Y})=0$, one has $\rho=0$, which implies $Y''\in\omega$ because $Y''$ is a direct summand of $X$. So  $\delta''=0$ as $\delta''\in\mathbb{E}(\mathcal{X}, \mathcal{Y})$.
By assumption $x'$ is a splitting inflation, thus $y'$ is a splitting deflation, and hence $\mathcal{B}(X, y'): \mathcal{B}(X, B')\longrightarrow\mathcal{B}(X, W)$ is surjective.
Applying the functor $\mathcal{B}(X, -)$ to the second and the third column of the above commutative diagram,
one sees that $\mathcal{B}(X, y): \mathcal{B}(X, B)\rightarrow \mathcal{B}(X, X)$ is surjective by Snake Lemma. Therefore, $x$ is a splitting inflation. Thus $x\in{\rm TCoFib}_\omega$.

\vskip5pt

Using ${\rm Weq}_\omega = {\rm TFib}_\omega \circ {\rm TCoFib}_\omega$  and ${\rm TCoFib}_\omega\circ{\rm TFib}_\omega\subseteq{\rm Weq}_\omega,$
one can see that ${\rm Weq}_\omega$ is closed under compositions :
\begin{align*}{\rm Weq}_\omega\circ {\rm Weq}_\omega & = {\rm TFib}_\omega \circ {\rm TCoFib}_\omega \circ {\rm TFib}_\omega \circ {\rm TCoFib}_\omega \subseteq {\rm TFib}_\omega \circ {\rm Weq}_\omega \circ {\rm TCoFib}_\omega\\
&= {\rm TFib}_\omega \circ {\rm TFib}_\omega \circ {\rm TCoFib}_\omega \circ {\rm TCoFib}_\omega \subseteq {\rm TFib}_\omega \circ {\rm TCoFib}_\omega = {\rm Weq}_\omega\end{align*}
by Lemma \ref{pro2}.\end{proof}

\begin{lem}\label{lem6} \ Let $\mathcal{B}=(\mathcal{B}, \mathbb{E},\mathfrak{s})$ be a weakly idempotent complete extriangulated category,
let $\mathcal{X}$ and $\mathcal{Y}$ be full subcategories of $\mathcal{B}$ which are closed under extensions, direct summands, and isomorphisms.
If $f:A\rightarrow B$ is a morphism in ${\rm Weq}_\omega$, then the morphism $(f, t): A\oplus W\rightarrow B$ is in ${\rm TFib}_\omega$ for any right $\omega$-approximation morphism $t: W\rightarrow B$.
\end{lem}

\begin{proof} \ Since $f:A\rightarrow B$ is in ${\rm Weq}_\omega$, by construction there is a morphism $(f, t'): A\oplus W'\rightarrow B$ in ${\rm TFib}_\omega$ with $W'\in\omega$ and ${\rm CoCone}(f, t')\in\mathcal{Y}$. Let $t: W\rightarrow B$ be any right $\omega$-approximation. Then there is a morphism $s: W'\rightarrow W$ such that $t'=ts$. Since $(f, t')$ is a deflation and $(f, t')=(f, t){\tiny \begin{pmatrix}1&0\\0&s\end{pmatrix}}$,
it follows from the weakly idempotent completeness of $\mathcal B$ that $(f, t): A\oplus W\rightarrow B$ is also a deflation. It remains to prove ${\rm CoCone}(f, t)\in\mathcal{Y}$.
Note that  ${\tiny \begin{pmatrix}1&0\\0&1\\0&0\end{pmatrix}}: A\oplus W'\rightarrow A\oplus W'\oplus W$ is a splitting inflation with ${\rm Cone}{\tiny \begin{pmatrix}1&0\\0&1\\0&0\end{pmatrix}}=W\in\omega$.
Since  $(f, t')=(f, t', t){\tiny \begin{pmatrix}1&0\\0&1\\0&0\end{pmatrix}}$ is a deflation,
 $(f, t', t)$ is a deflation. Therefore, there is a commutative diagram with $\mathbb{E}$-triangles
$$\xymatrix{{\rm CoCone}(f, t')\ar[r]^i\ar[d]&A\oplus W'\ar[r]^-{(f, t')}\ar[d]_{\tiny \begin{pmatrix}1&0\\0&1\\0&0\end{pmatrix}}&B\ar@{=}[d]\ar@{-->}[r]&\\
{\rm CoCone}(f, t', t)\ar[r]^{i'}\ar[d]&A\oplus W'\oplus W\ar[r]^-{(f, t', t)}\ar[d]&B\ar@{-->}[r]&\\
W\ar@{=}[r]\ar@{-->}[d]&W\ar@{-->}[d]&&\\
&&&}$$
by \cite[Proposition 3.17]{NP}. Hence ${\rm CoCone}(f, t', t)\in\mathcal{Y}$ because $\mathcal{Y}$ is closed under extensions. Since $(f, t)=(f, t', t){\tiny \begin{pmatrix}1&0\\0&0\\0&1\end{pmatrix}}$ and $(f, t', t)=(f, t){\tiny \begin{pmatrix}1&0&0\\0&s&1\end{pmatrix}}$, there is a  commutative diagram of $\mathbb{E}$-triangles
$$\xymatrix{{\rm CoCone}(f, t)\ar[r]\ar@{-->}[d]_g&A\oplus W\ar[r]^{(f, t)}\ar[d]_{\tiny\begin{pmatrix}1&0\\0&0\\0&1\end{pmatrix}}&B\ar@{=}[d]\ar@{-->}[r]&\\
{\rm CoCone}(f, t', t)\ar[r]\ar@{-->}[d]_h&A\oplus W'\oplus W\ar[r]^{\qquad(f, t', t)}\ar[d]_{\tiny\begin{pmatrix}1&0&0\\0&s&1\end{pmatrix}}&B\ar@{=}[d]\ar@{-->}[r]&\\
{\rm CoCone}(f,t)\ar[r]&A\oplus W\ar[r]^{(f, t)}&B\ar@{-->}[r]&}$$
by (ET3)$^{\rm op}$. Note that ${\tiny\begin{pmatrix}1&0&0\\0&s&1\end{pmatrix}}
{\tiny\begin{pmatrix}1&0\\0&0\\0&1\end{pmatrix}}={\tiny\begin{pmatrix}1&0\\0&1\end{pmatrix}}$.  So we have a commutative diagram of $\mathbb{E}$-triangles
$$\xymatrix@R=0.5cm{{\rm CoCone}(f,t)\ar[r]\ar@{-->}[d]_{hg}&A\oplus W\ar[r]^{(f, t)}\ar@{=}[d]&B\ar@{=}[d]\ar@{-->}[r]&\\
{\rm CoCone}(f,t)\ar[r]&A\oplus W\ar[r]^{(f, t)}&B\ar@{-->}[r]&.}$$
Therefore, $hg$ is an isomorphism, which implies that ${\rm CoCone}(f,t)$ is a direct summand of ${\rm CoCone}(f, t', t)$, thus ${\rm CoCone}(f,t)\in\mathcal{Y}$.  This completes the proof.
\end{proof}

\begin{lem}\label{lem7} \ Assume that $\mathcal{B}=(\mathcal{B}, \mathbb{E}, \mathfrak{s})$ is a weakly idempotent complete extriangulated category. Let $(\mathcal{X}, \mathcal{Y})$ be a hereditary cotorsion pair,
such that $\omega=\mathcal{X}\cap \mathcal{Y}$ is contravariantly finite in $\mathcal{B}$. Let $f:A\rightarrow B$ and $g:B\rightarrow C$ be morphisms in $\mathcal{B}$ with $f\in{\rm TFib}_\omega$ and $gf\in{\rm Weq}_\omega$. Then $g\in{\rm Weq}_\omega$.
\end{lem}

\begin{proof} Let $t:W\rightarrow C$ be a right $\omega$-approximation of $C$. By Lemma \ref{lem6} that $(gf, t): A\oplus W\rightarrow C$ is in ${\rm TFib}_{\omega}$, i.e., $(gf, t)$ is a deflation with ${\rm CoCone}(gf, t)\in\mathcal{Y}$.
By $(gf, t)=(g, t){\tiny\begin{pmatrix}f&0\\0&1\end{pmatrix}}$ one sees that $(g, t):B\oplus W\rightarrow C$ is a deflation. Then there is a commutative diagram of $\mathbb{E}$-triangles
$$\xymatrix@R=0.7cm{{\rm CoCone}{\tiny\begin{pmatrix}f&0\\0&1\end{pmatrix}}\ar[d]\ar@{=}[r]&{\rm CoCone}{\tiny\begin{pmatrix}f&0\\0&1\end{pmatrix}}\ar[d]&&\\
{\rm CoCone}(gf,  t)\ar[r]\ar[d]&A\oplus W\ar[r]^-{(gf,t)}\ar[d]_{\tiny\begin{pmatrix}f&0\\0&1\end{pmatrix}}&C\ar@{=}[d]\ar@{-->}[r]&\\
{\rm CoCone}(g,t)\ar[r]\ar@{-->}[d]&B\oplus W\ar[r]^-{(g, t)}\ar@{-->}[d]&C\ar@{-->}[r]&\\
&&&}$$
by \cite[Remark 2.22]{NP}. Since ${\rm CoCone}{\tiny\begin{pmatrix}f&0\\0&1\end{pmatrix}}={\rm CoCone}(f)\in\mathcal{Y}$, which implies ${\rm CoCone}(g, t)\in\mathcal{Y}$ because $\mathcal{Y}$ is closed under cones of inflations. So $(g, t)\in{\rm TFib}_\omega$, and we have $g\in{\rm Weq}_\omega$ by definition.
\end{proof}

\begin{lem}\label{lem8} \ Assume that $\mathcal{B}=(\mathcal{B}, \mathbb{E}, \mathfrak{s})$ is a weakly idempotent complete extriangulated category. Let $(\mathcal{X}, \mathcal{Y})$ be a hereditary cotorsion pair,
such that $\omega=\mathcal{X}\cap \mathcal{Y}$ is contravariantly finite in $\mathcal{B}$. Let $f:A\rightarrow B$ and $g:B\rightarrow C$ be morphisms in $\mathcal{B}$ such that $f\in{\rm Weq}_\omega$ and $gf\in{\rm Weq}_\omega$. Then $g\in{\rm Weq}_\omega$.
\end{lem}

\begin{proof} Since $f\in{\rm Weq}_\omega$, there is a morphism $(f, t): A\oplus W\rightarrow B$ with $W\in\omega$ and $(f, t)\in{\rm TFib}_\omega$. To prove $g\in {\rm Weq}_\omega$, it suffices to prove $(gf, gt)\in{\rm Weq}_\omega$ by the commutative diagram

$$\xymatrix@R=0.5cm{A\oplus W\ar[rr]^{(gf, gt)}\ar[rd]_{(f, t)}&&C\\
&B\ar[ur]_g&}$$
and Lemma \ref{lem7}.

Let $t':W'\rightarrow C$ be a right $\omega$-approximation of $C$. It follows from Lemma \ref{lem6} that $(gf, t'): A\oplus W'\rightarrow C$ is in ${\rm TFib}_\omega$, i.e., $(gf, t')$ is a deflation with ${\rm CoCone}(gf, t')\in\mathcal{Y}$. Observe that ${\tiny \begin{pmatrix}1&0\\0&1\\0&0\end{pmatrix}}: A\oplus W'\rightarrow A\oplus W'\oplus W$ is a splitting inflation with ${\rm Cone}{\tiny \begin{pmatrix}1&0\\0&1\\0&0\end{pmatrix}}=W\in\omega$.
Since $(gf, t')=(gf, t', gt){\tiny \begin{pmatrix}1&0\\0&1\\0&0\end{pmatrix}}$ is a deflation, where $(gf, t', gt): A\oplus W'\oplus W\rightarrow C$, it follows from the weakly idempotent completeness
that $(gf, t', gt)$ is a deflation. Therefore, there is a commutative diagram with $\mathbb{E}$-triangles
$$\xymatrix{{\rm CoCone}(gf, t')\ar[r]^i\ar[d]&A\oplus W'\ar[rr]^-{(gf, t')}\ar[d]_{\tiny \begin{pmatrix}1&0\\0&1\\0&0\end{pmatrix}}&&C\ar@{=}[d]\ar@{-->}[r]&\\
{\rm CoCone}(gf, t', gt)\ar[r]^{i'}\ar[d]&A\oplus W'\oplus W\ar[rr]^-{(gf, t', gt)}\ar[d]&&C\ar@{-->}[r]&\\
W\ar@{=}[r]\ar@{-->}[d]&W\ar@{-->}[d]&&\\
&&&}$$
by \cite[Proposition 3.17]{NP}. Since ${\rm CoCone}(gf, t')\in\mathcal{Y}$ and $W\in\mathcal{Y}$, which implies ${\rm CoCone}(gf, t', gt)\in\mathcal{Y}$, and hence $(gf, t', gt)\in{\rm TFib}_\omega$. Therefore, $(gf, gt)\in{\rm Weq}_\omega$ by the commutative diagram
$$\xymatrix{A\oplus W\ar[rr]^{(gf, gt)}\ar[rd]_{\tiny \begin{pmatrix}1&0\\0&0\\0&1\end{pmatrix}}&&C\\
&A\oplus W'\oplus W.\ar[ur]_{(gf, t', gt)}&}$$
\noindent This completes the proof.
\end{proof}

\begin{lem}\label{lem9} \ Assume that $\mathcal{B}=(\mathcal{B}, \mathbb{E}, \mathfrak{s})$ is a weakly idempotent complete extriangulated category.
Let $(\mathcal{X}, \mathcal{Y})$ be a hereditary cotorsion pair. Let $f: A\rightarrow B$ and $g:B\rightarrow C$ be morphisms in $\mathcal{B}$ with $gf\in{\rm Weq}_\omega$, $f\in{\rm CoFib}_\omega$ and $g\in{\rm TFib}_\omega$. Then $f\in{\rm Weq}_\omega$.
\end{lem}

\begin{proof} We first show that $f$ is a splitting inflation. Since $f\in{\rm CoFib}_\omega$, $f$ is an inflation with ${\rm Cone}(f)\in\mathcal{X}$. As $h:=gf\in{\rm Weq}_\omega$, there is a deflation $(h, t): A\oplus W\rightarrow C$ with $W\in\omega$ and ${\rm CoCone}(h,t)\in\mathcal{Y}$. Consider the commutative diagram with $\mathbb{E}$-triangles
$$\xymatrix{&A\ar[r]^f\ar[d]_{\binom{1}{0}}&B\ar[r]\ar[d]^g&{\rm Cone}(f)\ar@{-->}[r]&\\
{\rm CoCone}(h, t)\ar[r]&A\oplus W\ar[r]^-{(h,t)}&C\ar@{-->}[r]&.&}$$
Since $\mathbb{E}({\rm Cone}(f), {\rm CoCone}(h, t))=0$, there exists a morphism $\binom{s_1}{s_2}: B\rightarrow A\oplus W$ such that $1_A=s_1f$ by the Extension-Lifting Lemma \ref{lem3}. This shows that $f$ is a splitting inflation.

It remains to prove that ${\rm Cone}(f)\in\mathcal{Y}$. Since $f$ is a splitting inflation, one can write $h=gf$ as
$$\xymatrix@R=0.5cm{A\ar[rr]^{h}\ar[rd]_{f=\binom{1}{0}}&&C\\
&A\oplus X\ar[ur]_{g=(h, t')}&}$$
where $X:={\rm Cone}(f)\in\mathcal{X}$ and $g=(h, t')\in{\rm TFib}_\omega$. Since $(h, t')=(h, t, t'){\tiny \begin{pmatrix}1&0\\0&0\\0&1\end{pmatrix}}$ is a deflation, where $(h, t, t'): A\oplus W\oplus X\rightarrow C$,   it follows that $(h, t, t')$ is a deflation. Therefore, there is a commutative diagram with $\mathbb{E}$-triangles
$$\xymatrix{{\rm CoCone}(h, t')\ar[r]\ar[d]&A\oplus X\ar[rr]^{(h, t')}\ar[d]^{\tiny \begin{pmatrix}1&0\\0&0\\0&1\end{pmatrix}}&&C\ar@{=}[d]\ar@{-->}[r]&\\
{\rm CoCone}(h, t, t')\ar[r]^{\tiny \begin{pmatrix}k_1\\k_2\\k_3\end{pmatrix}}\ar[d]&A\oplus W\oplus X\ar[rr]^{\qquad (h, t, t')}\ar[d]&&C\ar@{-->}[r]&\\
W\ar@{=}[r]\ar@{-->}[d]&W\ar@{-->}[d]&&\\
&&&}$$
by \cite[Proposition 3.17]{NP}. Since ${\rm CoCone}(h, t')\in\mathcal{Y}$ and $W\in\omega$, we have ${\rm CoCone}(h, t, t')\in\mathcal{Y}$. The commutative square
 $$\xymatrix{{\rm CoCone}(h, t, t')\ar[r]^{\qquad k_3}\ar[d]_{\binom{k_1}{k_2}}&X\ar[d]^{-t'}\\
 A\oplus W\ar[r]^{(h, t)}&C}$$
is a homotopy cartesian in \cite[Definition 1.1]{HXZ} because there is an $\mathbb{E}$-triangle
 $$\xymatrix{{\rm CoCone}(h, t, t')\ar[r]^{\tiny \begin{pmatrix}k_1\\k_2\\k_3\end{pmatrix}}&A\oplus W\oplus X\ar[r]^{\qquad (h, t, t')}&C\ar@{-->}[r]&.}$$
Hence there is a commutative diagram of $\mathbb{E}$-triangles
$$\xymatrix{{\rm CoCone}(h,t)\ar@{=}[d]\ar[r]&{\rm CoCone}(h, t, t')\ar[r]^{\qquad k_3}\ar[d]_{\tiny \begin{pmatrix}k_1\\k_2\end{pmatrix}}&X\ar[d]^{-t'}\ar@{-->}[r]&\\
{\rm CoCone}(h,t)\ar[r]& A\oplus W\ar[r]^{(h, t)}&C\ar@{-->}[r]&}$$
by the dual of \cite[Lemma 3.1(2)]{HXZ}. Since ${\rm CoCone}(h,t)\in\mathcal{Y}$ and ${\rm CoCone}(h,t, t')\in\mathcal{Y}$, we have $X\in\mathcal{Y}$ as $\mathcal{Y}$ is closed under cones of inflations. This completes the proof.
\end{proof}
\begin{lem}\label{lem10} \ Assume that $\mathcal{B}=(\mathcal{B}, \mathbb{E}, \mathfrak{s})$ is a weakly idempotent complete extriangulated category.
Let $(\mathcal{X}, \mathcal{Y})$ be a hereditary cotorsion pair, such that $\omega=\mathcal{X}\cap \mathcal{Y}$ is contravariantly finite in $\mathcal{B}$. Let $f: A\rightarrow B$ and $g:B\rightarrow C$ be morphisms in $\mathcal{B}$ with $gf\in{\rm Weq}_\omega$ and $g\in{\rm Weq}_\omega$. Then $f\in{\rm Weq}_\omega$.
\end{lem}
\begin{proof} The proof is similar to that of \cite[Lemma 3.10]{CLZ} (this needs the factorization axiom, and hence needs the condition that $\omega=\mathcal{X}\cap \mathcal{Y}$ is contravariantly finite in $\mathcal{B}$),  by noting that Lemma \ref{lem9} holds.\end{proof}

Now, Lemmas \ref{lem5}, \ref{lem8} and \ref{lem10} yield the following \textbf{Two out of three axiom}.

\begin{prop}\label{prop4} \ Assume that $\mathcal{B}=(\mathcal{B}, \mathbb{E}, \mathfrak{s})$ is a weakly idempotent complete extriangulated category. Let $(\mathcal{X}, \mathcal{Y})$ be a hereditary complete cotorsion pair,
such that $\omega=\mathcal{X}\cap \mathcal{Y}$ is contravariantly finite in $\mathcal{B}$. Let $f: A\rightarrow B$ and $g:B\rightarrow C$ be morphisms in $\mathcal{B}$. If two of three morphisms $f, g, gf$ are in ${\rm Weq}_\omega$, then so does the third.
\end{prop}

In the following, we prove that ${\rm CoFib}_\omega, {\rm Fib}_\omega$ and ${\rm Weq}_\omega$ are closed under retracts. Suppose that $g: A\rightarrow B$ is a retract of $f: C\rightarrow D$, i.e., one has a commutative diagram of morphisms
$$\xymatrix@R=0.5cm{A\ar[r]^{a}\ar[d]_g&C\ar[r]^{c}\ar[d]^f&A\ar[d]^g\\
B\ar[r]^{b}&D\ar[r]^{d}&B}$$
with $ca=1_{A}$ and $db=1_{B}$.

\begin{lem}\label{lem12}  \ Let $\mathcal{B}=(\mathcal{B}, \mathbb{E},\mathfrak{s})$ be a weakly idempotent complete extriangulated category, let $\mathcal{X}$ and $\mathcal{Y}$ be full additive subcategories of $\mathcal{B}$
which are closed under direct summands and isomorphisms. Then ${\rm CoFib}_\omega$, ${\rm TCoFib}_\omega$, and ${\rm TFib}_\omega$ are  closed under retracts.\end{lem}

\begin{proof} Let $f\in{\rm CoFib}_\omega$, i.e., $f$ is an inflation with ${\rm Cone}(f)\in\mathcal{X}$. Since $bg=fa$ is an inflation, $g$ is an inflation. It follows from (ET3) that there exist morphisms $m$ and $n$ which make the following diagram commutative
$$\xymatrix@R=0.5cm{A\ar[r]^{a}\ar[d]_g&C\ar[r]^{c}\ar[d]^f&A\ar[d]^g\\
B\ar[r]^{b}\ar[d]_{p}&D\ar[r]^{d}\ar[d]^{q}&B\ar[d]^{p}\\
{\rm Cone}(g)\ar@{-->}[r]^{m}\ar@{-->}[d]&{\rm Cone}(f)\ar@{-->}[r]^{n}\ar@{-->}[d]&{\rm Cone}(g),\ar@{-->}[d]&\\&&}$$
where all columns are $\mathbb{E}$-triangles. Hence there is a commutative diagram with $\mathbb{E}$-triangles
$$\xymatrix@R=0.5cm{A\ar[r]^-g\ar@{=}[d]&B\ar[r]^-{p}\ar@{=}[d]&{\rm Cone}(g)\ar[d]^{nm}\ar@{-->}[r]&\\
A\ar[r]^g&B\ar[r]^-{p}&{\rm Cone}(g)\ar@{-->}[r]&}$$
Thus $nm$ is an isomorphism, which implies that ${\rm Cone}(g)$ is a direct summand of ${\rm Cone}(f)$. Hence ${\rm Cone}(g)\in\mathcal{X}$ and $g\in{\rm CoFib}_\omega$.

 Similarly for the other cases. \end{proof}

\begin{lem}\label{lem13}  \ Let $\mathcal{B}=(\mathcal{B}, \mathbb{E},\mathfrak{s})$ be an extriangulated category, let $\mathcal{X}$ and $\mathcal{Y}$ be full additive subcategories of $\mathcal{B}$
which are closed under direct summands and isomorphisms. Then ${\rm Fib}_\omega$  is closed under retracts.\end{lem}

\begin{proof} The proof is similar to  that of \cite[Step 2, p.18]{CLZ}.
\end{proof}

\begin{lem}\label{lem14} \ Assume that $\mathcal{B}=(\mathcal{B}, \mathbb{E}, \mathfrak{s})$ is a weakly idempotent complete extriangulated category. Let $(\mathcal{X}, \mathcal{Y})$ be a hereditary complete cotorsion pair,
such that $\omega=\mathcal{X}\cap \mathcal{Y}$ is contravariantly finite in $\mathcal{B}$. Then ${\rm Weq}_\omega$  is closed under retracts.\end{lem}

\begin{proof} Let $g=g_1\oplus g_2: A_1\oplus A_2\rightarrow B_1\oplus B_2$ and suppose $g\in{\rm Weq}_\omega$. Then $g_n$ admits a factorization  $g_n=p_ni_n$, $i_n\in{\rm CoFib}_\omega$, $p_n\in{\rm TFib}_\omega$ for $n=1, 2$ by Proposition \ref{pro3}. Therefore, $g=g_1\oplus g_2=(p_1\oplus p_2)\circ (i_1\oplus i_2)$. Since $p_1, p_2\in{\rm TFib}_\omega$, one has $p_1\oplus p_2\in{\rm TFib}_\omega\subseteq{\rm Weq}_\omega$. So we have $i_1\oplus i_2\in{\rm Weq}_\omega$ by Lemma \ref{lem10}. Since $i_1, i_2\in{\rm CoFib}_\omega$, we have $i_1\oplus i_2\in{\rm CoFib}_\omega$. Therefore, $i_1\oplus i_2\in{\rm CoFib}_\omega\cap{\rm Weq}_\omega={\rm TCoFib}_\omega$ by Proposition \ref{pro1}, which implies that $i_1, i_2\in{\rm TCoFib}_\omega$ by Lemma \ref{lem12}. Therefore, we have $g_1, g_2\in{\rm Weq}_\omega$, which implies that ${\rm Weq}_\omega$  is closed under retracts.
\end{proof}

{\bf Proof of the Retract axiom}. Now, the Retract axiom follows from Lemmas \ref{lem12}-\ref{lem14}.  \hfill$\Box$

\vspace{2mm}
Finally, we have the following \textbf{Lifting axiom}. Its proof is similar to that of Lifting axiom in \cite[Subsection 3.5]{CLZ}.

\begin{prop}\label{pro5} \ Assume that $\mathcal{B}=(\mathcal{B}, \mathbb{E}, \mathfrak{s})$ is a weakly idempotent complete extriangulated category. Let $\mathcal{X}, \mathcal{Y}$ be full additive subcategories of $\mathcal{B}$ which are closed under direct summands and isomorphisms, and $\omega=\mathcal{X}\cap \mathcal{Y}$. If $\mathbb{E}(\mathcal{X}, \mathcal{Y})=0$, then

\vspace{1mm}

{\rm (1)} \ ${\rm TCoFib}_\omega$ satisfies the left lifting property with respect to ${\rm Fib}_\omega$.
\vspace{1mm}

{\rm (2)} \ ${\rm TFib}_\omega$ satisfies the left lifting property with respect to ${\rm CoFib}_\omega$.
\end{prop}

\textbf{Up to now the ``if" part of Theorem \ref{thmA} is proved.}

\subsection{Hereditary complete cotorsion pair arising from a model structure}

The aim of this subsection is to prove the ``only if" part of Theorem \ref{thmA} as follows.

\begin{thm}\label{thmC} Let $\mathcal{B}$ be a weakly idempotent complete extriangulated category, let $\mathcal{X}, \mathcal{Y}$ be full additive subcategories of $\mathcal{B}$ which are closed under direct summands and isomorphisms, and let $\omega=\mathcal{X}\cap \mathcal{Y}$. If $({\rm CoFib}_\omega, {\rm Fib}_\omega, {\rm Weq}_\omega)$ is a model structure on $\mathcal{B}$, then
$(\mathcal{X, Y})$ is a hereditary complete cotorsion pair on $\mathcal{B}$, and $\omega$ is contravariantly finite in $\mathcal{B}$. Moreover, the class $\mathcal{C}_\omega$ of cofibrant objects is $\mathcal{X}$, the class $\mathcal{F}_\omega$ of fibrant objects is $\mathcal{B}$, the class $\mathcal{W}_\omega$ of trivial objects is $\mathcal{Y}$; and the homotopy category ${\rm Ho}(\mathcal{B})$ is equivalent to the additive quotient $\mathcal{X}/\omega$.
\end{thm}

To prove Theorem \ref{thmC}, we need some preparations.

\begin{lem}\label{lem3.2.1} Assume that $(\mathcal{X}, \mathcal{Y})$ is a complete cotorsion pair on an extriangulated category $\mathcal{B}$ with $\omega=\mathcal{X}\cap\mathcal{Y}$. If  $({\rm CoFib}_\omega, {\rm Fib}_\omega, {\rm Weq}_\omega)$ is a model structure as given in {\rm (3.1)}, then the homotopy category ${\rm Ho}(\mathcal{B})$ is equivalent to the additive quotient $\mathcal{X}/\omega$.
\end{lem}
\begin{proof} First, we need to apply Theorem \ref{thm} to show that ${\rm Ho}(\mathcal{B})$ is equivalent to the quotient category $\pi\mathcal{B}_{cf}$. For this purpose, we need to verify the five conditions in Condition \ref{condition}. It is clear that conditions (i) and (ii) hold as $\mathcal{B}$ is an additive category. It remains to prove conditions (iv) and (v) (we prefer to include a short justification, although it can be omitted by \cite[Theorem 3.2]{Egger}).

 For any trivial cofibration $i: A\rightarrow B$ and any morphism $u: A\rightarrow C$,  there exists a splitting  $\mathbb{E}$-triangle $\xymatrix{A\ar[r]^u&B\ar[r]&W\ar@{-->}[r]^\delta&}$ such that $W\in\omega$ by definition of trivial cofibrations. Hence there is a commutative diagram
$$\xymatrix@R=0.5cm{A\ar[r]^i\ar[d]_u&B\ar[d]\ar[r]
 &W\ar@{-->}[r]^\delta\ar@{=}[d]&\\C\ar[r]^{i'}&D\ar[r]&W\ar@{-->}[r]^{u_*\delta}&}$$
such that the left square is a weak push-out. Since $\delta=0$, one has $u_*\delta=0$. So the second row is a splitting $\mathbb{E}$-triangle with $W\in\omega$, that is, $i'$ is a trivial cofibration. This shows that condition (iv) holds.

 For any trivial fibration $p: C\rightarrow D$ and any morphism $u: B\rightarrow D$,  there exists an  $\mathbb{E}$-triangle $\xymatrix{Y\ar[r]&C\ar[r]^p&D\ar@{-->}[r]^\delta&}$ such that $Y\in\mathcal{Y}$ by definition of trivial fibrations. Hence there is a commutative diagram
$$\xymatrix@R=0.5cm{Y\ar[r]\ar@{=}[d]&A\ar[d]
\ar[r]^{p'}&B\ar@{-->}[r]^{u^*\delta}\ar[d]^u&\\Y\ar[r]&C\ar[r]^p&D\ar@{-->}[r]^{\delta}&}
$$
such that the right square is a weak pull-back.  Hence $p'$ is a trivial fibration by definition. This shows that condition (v) holds.

Next, we will prove the equality $\pi\mathcal{B}_{cf}=\mathcal{X}/\omega$. For the model structure $({\rm CoFib}_\omega, {\rm Fib}_\omega, {\rm Weq}_\omega)$, it is clear that $\mathcal{B}_{cf}=\mathcal{X}$. Let $f, g: A\rightarrow B$ be morphisms with $A, B\in\mathcal{X}$. It is suffices to show that $f\overset{l}{\sim}g\Leftrightarrow f-g$ factors through an object in $\omega$.

If $f\overset{l}{\sim} g$, then by the definition of the left homotopy relation $\overset{l}{\sim}$, one has a commutative diagram
\begin{center}$\xymatrix@R=0.5cm@C=3em{A\oplus A\ar[d]_{(1,1)}\ar[rr]^{(f, g)}\ar[rrd]^{(\partial_1, \partial_2)}&&B\\
A&&\widetilde{A}\ar[ll]_s\ar[u]_h}$
\end{center}
with $s\in{\rm Weq}_\omega$. By definition, there is a deflation $(s, t): \widetilde{A}\oplus W\overset{(s, t)}{\longrightarrow} A$ with $W\in\omega$ and ${\rm CoCone}(s, t)\in\mathcal{Y}$. Hence $(s, t)\in{\rm TFib}_\omega$ and there is a commutative diagram
\begin{center}$\xymatrix@R=0.5cm@=3em{A\oplus A\ar[d]_{(1,1)}\ar[rr]^{(f, g)}\ar[rrd]^{\tiny\begin{pmatrix}\partial_1&\partial_2\\0&0\end{pmatrix}}&&B\\
A&&\widetilde{A}\oplus W\ar[ll]_{(s,t)}\ar[u]_{(h,0)}}$
\end{center}
Therefore, $f-g=(h, 0){\tiny\begin{pmatrix}\partial_1-\partial_2\\0\end{pmatrix}}$ and $(s, t){\tiny\begin{pmatrix}\partial_1-\partial_2\\0\end{pmatrix}}=0$. We claim that ${\tiny\begin{pmatrix}\partial_1-\partial_2\\0\end{pmatrix}}$ factors through an object in $\omega$.
Since $(\mathcal{X}, \mathcal{Y})$ is a complete cotorsion pair, there is an $\mathbb{E}$-triangle $\xymatrix{A\ar[r]^i&I\ar[r]&X\ar@{-->}[r]&}$ with $I\in\mathcal{Y}$ and $X\in\mathcal{X}$. Then $i\in{\rm CoFib}_\omega$. Since $\mathcal{X}$ is closed under extensions, one has $I\in\mathcal{X}\cap\mathcal{Y}=\omega$. Consider a commutative diagram
\begin{center}$\xymatrix@R=0.5cm@C = 1.2cm{A\ar[r]^{\tiny\begin{pmatrix}\partial_1-\partial_2\\0\end{pmatrix}\ \ }\ar[d]_i&\widetilde{A}\oplus W\ar[d]^{(s, t)}\\I\ar@{-->}[ur]\ar[r]^0&A}$\end{center}
By the lifting axiom, ${\tiny\begin{pmatrix}\partial_1-\partial_2\\0\end{pmatrix}}$ factors through an object $I$ in $\omega$.

Conversely, if $f-g$ factors through an object $W$ in $\omega$, say $f-g=v\circ u$ with $A\overset{u}{\longrightarrow}W\overset{v}{\longrightarrow}B$, then there is a commutative diagram
\begin{center}$\xymatrix@=3em{A\oplus A\ar[d]_{(1,1)}\ar[rr]^{(f, g)}\ar[rrd]^{\tiny\begin{pmatrix}1&1\\u&0\end{pmatrix}}&&B\\
A&&A\oplus W\ar[ll]_{(1, 0)}\ar[u]_{(g, v)}}$
\end{center}
where $(1, 0)\in{\rm TFib}_\omega \subseteq {\rm Weq}_\omega$. Thus $f\overset{l}{\sim} g$. This completes the proof.
\end{proof}

{\bf Proof of Theorem \ref{thmC}.} By definition one easily sees that $\mathcal{C}_\omega=\mathcal{X}$ and $\mathcal{F}_\omega=\mathcal{B}$. We claim that $\mathcal{W}_\omega=\mathcal{Y}$. If $Y\in\mathcal{Y}$, since $(0, 0): Y\oplus 0\rightarrow 0$ is a deflation with $0\in\omega$ and ${\rm CoCone}(0,0)=Y\in\mathcal{Y}$, hence $0: Y\rightarrow 0$ is a weak equivalence by definition, i.e., $Y\in\mathcal{W}_\omega$. Conversely, if $Y\in\mathcal{W}_\omega$, then $0: Y\rightarrow 0$ is a weak equivalence. Hence there is a deflation $(0, 0): Y\oplus W\rightarrow 0$ with ${\rm CoCone}(0, 0)=Y\oplus W\in \mathcal{Y}$, which implies $Y\in \mathcal{Y}$. So we have $\mathcal{W}_\omega=\mathcal{Y}$. Thus ${\rm T}\mathcal{C}_\omega=\mathcal{C}_\omega\cap \mathcal{Y}=\omega$ and ${\rm T}\mathcal{F}_\omega=\mathcal{F}_\omega\cap\mathcal{W}_\omega=\mathcal{Y}$.

Next we prove that $(\mathcal{X, Y})$ is a hereditary complete cotorsion pair. Let $\delta\in\mathbb{E}(X, Y)$ with $X\in\mathcal{X}$ and $Y\in\mathcal{Y}$. Realize it as an $\mathbb{E}$-triangle $\xymatrix{Y\ar[r]^x&A\ar[r]^y&X\ar@{-->}[r]^\delta&.}$ Since $Y$ is in $\mathcal{Y}$, it follows that $y$ is a trivial fibration. Note that $X\in\mathcal{X}=\mathcal{C}_\omega$. Then the morphism $1_X: X\rightarrow X$ factors through $y$ by \cite[Lemma 4.2]{CLZ}. Thus $\delta$ splits by \cite[Corollary 3.5]{NP}, and so $\mathbb{E}(X, Y)=0$.

Let $C\in\mathcal{B}$ be any object, and let $\xymatrix@=2em{0\ar[r]&X_C\ar[r]^{f}&C}$ be a factorization of  $0: 0\rightarrow C$ into a cofibration followed by a trivial fibration. Then $X_C\in \mathcal{C}_\omega=\mathcal{X}$ by definition.
Since $f$ is a trivial fibration, there exists an $\mathbb{E}$-triangle $\xymatrix{Y_C\ar[r]&X_C\ar[r]^{f}&C\ar@{-->}[r]&}$ with $Y_C\in\mathcal{Y}$. Dually, there exists an $\mathbb{E}$-triangle $\xymatrix{C\ar[r]&Y^C\ar[r]&X^C\ar@{-->}[r]&}$ with $X^C\in\mathcal{X}$ and $Y^C\in\mathcal{Y}$. Hence $(\mathcal{X, Y})$ is a complete cotorsion pair by  Fact \ref{fact2.18}.

Thanks to Lemma \ref{lem3.2.1}, we have  ${\rm Ho}(\mathcal{B})\cong\mathcal{X}/\omega$.

We can prove that $\mathcal{Y}$ is closed under cones of inflations by a similar argument of \cite[Proposition 1.2]{CLZ}, hence  $(\mathcal{X, Y})$ is a hereditary complete cotorsion pair by  Proposition \ref{pro0}.

Finally, we claim that $\omega$ is contravariantly finite. For any object $C\in\mathcal{B}$, there exists an $\mathbb{E}$-triangle $\xymatrix{Y_C\ar[r]&X_C\ar[r]^{f}&C\ar@{-->}[r]&}$ with $X_C\in\mathcal{X}$ and $Y_C\in\mathcal{Y}$. Hence $f$ is a trivial fibration by definition. Let $\xymatrix@=2em{0\ar[r]&F_C\ar[r]^{g}&X_C}$ be a factorization of $0: 0\rightarrow X_C$ into a trivial cofibration followed by a fibration. Then $F_C\in{\rm T}\mathcal{C}_\omega=\omega$. We claim that $fg: F_C\rightarrow C$ is a right $\omega$-approximation of $C$. Since fibrations are closed under compositions, $fg$ is a fibration. Let $t: W\rightarrow C$ be a morphism with $W\in \omega$. Consider the following diagram
$$\xymatrix{0\ar[r]\ar[d]&F_C\ar[d]^{fg}\\
W\ar[r]^t\ar@{-->}[ur]^s&C}$$
Since $0\rightarrow W$ is a trivial cofibration, by the lifting axiom, $t$ factors through $fg$. Hence $f g: F_C\rightarrow C$ is a right $\omega$-approximation of $C$. This completes the proof. \hfill$\Box$

\section{The Beligiannis-Reiten correspondence}

This section aims to give the Beligiannis-Reiten correspondence, a bijection between weakly projective model structures on a weakly idempotent complete extriangulated category $\mathcal{B}$ and hereditary complete cotorsion pairs such that the heart is contravariantly finite in $\mathcal{B}$.

\subsection{\bf Weakly projective model structures}

Let $({\rm CoFib}, {\rm Fib}, {\rm Weq})$ be a model structure on a pointed category $\mathcal{A}$ (i.e., a category with zero object). Put
\begin{align*} & \mathcal C: = \{\mbox{cofibrant objects}\}, \ \ \ \mathcal F: = \{\mbox{fibrant objects}\},
\ \ \ \mathcal W: = \{\mbox{trivial objects}\} \\
& {\rm T}\mathcal C: = \{\mbox{trivially cofibrant objects}\} , \ \ \ {\rm T}\mathcal F: = \{\mbox{trivially fibrant objects}\}.\end{align*}

By the Lifting axiom one can get the following fact (\cite[VIII, 1.1]{Beli-Reiten}).

\begin{lem}\label{lem4.1}Let $({\rm CoFib}, {\rm Fib}, {\rm Weq})$ be a model structure on pointed category $\mathcal{A}$. Then

$(1)$ If $p: B\rightarrow C$ is a trivial fibration $($respectively, a fibration$)$, then any morphism $h: X\rightarrow C$ factors through $p$, where $X\in\mathcal{C}$ $($respectively, $X\in{\rm T}\mathcal{C}$$)$.

$(2)$ If $i: A\rightarrow B$ is a trivial cofibration $($respectively, a cofibration$)$, then any morphism $g: A\rightarrow Y$ factors through $i$, where $Y\in\mathcal{F}$ $($respectively, $X\in{\rm T}\mathcal{F}$$)$.
\end{lem}

By Lemma \ref{lem4.1} one has the following lemma (\cite[VIII, 2.1]{Beli-Reiten}).

\begin{lem}\label{lem4.2}Let $({\rm CoFib}, {\rm Fib}, {\rm Weq})$ be a model structure on a pointed category $\mathcal{A}$. Then

$(1)$ The full subcategory $\mathcal{C}$ $($respectively, ${\rm T}\mathcal{C})$ a contravariantly finite in $\mathcal{A}$. Explicitly, for any object $A$ of $\mathcal{A}$, there is a right $\mathcal{C}$-approximation $($respectively, a right ${\rm T}\mathcal{C}$-approximation$)$ $f: C\rightarrow A$ with $f\in{\rm TFib}~($respectively, $f\in{\rm Fib})$.

$(2)$ The full subcategory $\mathcal{F}$ $($respectively, ${\rm T}\mathcal{F})$ a covariantly finite in $\mathcal{A}$. Explicitly, for any object $A$ of $\mathcal{A}$, there is a left $\mathcal{F}$-approximation $($respectively, a left ${\rm T}\mathcal{F}$-approximation$)$ $g: A\rightarrow F$ with $g\in{\rm TCoFib}~($respectively, $g\in{\rm CoFib})$.
\end{lem}

The following result is \cite[VIII, 3.2]{Beli-Reiten} in an abelian category and \cite[Lemma 4.4]{CLZ} in an exact category, with a slight difference.

\begin{lem}\label{lem4.3}Let $({\rm CoFib}, {\rm Fib}, {\rm Weq})$ be a model structure on an extriangulated category $\mathcal{B}$.

$(1)$ If $\{f\mid f$ is an inflation with ${\rm Cone}(f)\in\mathcal{C}\}\subseteq {\rm CoFib}$, then $\mathbb{E}(\mathcal{C}, {\rm T}\mathcal{F})=0$.

$(2)$ If ${\rm TFib}\subseteq \{f\mid f$ is a deflation with ${\rm CoCone}(f)\in{\rm T}\mathcal{F}\}$, then $^\perp{\rm T}\mathcal{F}\subseteq \mathcal{C}$.

$(3)$ If ${\rm CoFib}\subseteq \{f\mid f$ is an inflation with ${\rm Cone}(f)\in\mathcal{C}\}$, then $\mathcal{C}^\perp\subseteq{\rm T}\mathcal{F}$.

$(1')$ If $\{f\mid f$ is a deflation with ${\rm CoCone}(f)\in\mathcal{F}\}\subseteq {\rm Fib}$, then $\mathbb{E}({\rm T}\mathcal{C},\mathcal{F})=0$.

$(2')$ If ${\rm TCoFib}\subseteq \{f\mid f$ is an inflation with ${\rm Cone}(f)\in{\rm T}\mathcal{C}\}$, then ${\rm T}\mathcal{C}^\perp\subseteq \mathcal{F}$.

$(3')$ If ${\rm Fib}\subseteq \{f\mid f$ is a deflation with ${\rm CoCone}(f)\in\mathcal{F}\}$, then $^\perp\mathcal{F}\subseteq{\rm T}\mathcal{C}$.
\end{lem}

\begin{proof} By duality, it is sufficient to prove $(1')$-$(3')$.

$(1')$ For any $\mathbb{E}$-triangle $\xymatrix{A\ar[r]^i&B\ar[r]^d&C\ar@{-->}[r]&}$ with $A\in\mathcal{F}$ and $C\in{\rm T}\mathcal{C}$, then $d$ is a fibration by assumption. Thus ${\rm 1}_C: C\rightarrow C$ factors through $d$ by Lemma \ref{lem4.1}(1), i.e., $d$ is a splitting deflation, hence $\mathbb{E}({\rm T}\mathcal{C},\mathcal{F})=0$.

$(2')$ For any object $A\in{\rm T}\mathcal{C}^\perp$, there is a left $\mathcal{F}$-approximation $g: A\rightarrow F$ with $g\in{\rm TCoFib}$ by Lemma \ref{lem4.2}(2). Hence there is an $\mathbb{E}$-triangle $\xymatrix{A\ar[r]^g&F\ar[r]&X\ar@{-->}[r]&}$ with $X\in{\rm T}\mathcal{C}$ by assumption. Since $A\in{\rm T}\mathcal{C}^\perp$, this $\mathbb{E}$-triangle splits. Hence $A$ is a direct summand of $F$, and thus $A\in\mathcal{F}$.

$(3')$ The proof is similar to that of $(2')$.
\end{proof}

\begin{prop}\label{prop4.4} Let $({\rm CoFib}, {\rm Fib}, {\rm Weq})$ be a model structure on an extriangulated category $\mathcal{B}$. Then the following are equivalent:

$(1)$ ${\rm CoFib}=\{f\mid f$ is an inflation with ${\rm Cone}(f)\in\mathcal{C}\}$, and
\begin{center}
  ${\rm TFib}\subseteq \{f\mid f$ is a deflation with ${\rm CoCone}(f)\in{\rm T}\mathcal{F}\}$.
\end{center}

$(2)$ $\mathbb{E}(\mathcal{C}, {\rm T}\mathcal{F})=0, {\rm CoFib}\subseteq\{f\mid f$ is an inflation with ${\rm Cone}(f)\in\mathcal{C}\}$, and
\begin{center}
${\rm TFib}\subseteq \{f\mid f$ is a deflation with ${\rm CoCone}(f)\in{\rm T}\mathcal{F}\}$.
\end{center}

$(3)$ ${\rm CoFib}\subseteq\{f\mid f$ is an inflation with ${\rm Cone}(f)\in\mathcal{C}\}$, and
\begin{center}
${\rm TFib}= \{f\mid f$ is a deflation with ${\rm CoCone}(f)\in{\rm T}\mathcal{F}\}$.
\end{center}

$(4)$ ${\rm CoFib}=\{f\mid f$ is an inflation with ${\rm Cone}(f)\in\mathcal{C}\}$, and
\begin{center}
${\rm TFib}=\{f\mid f$ is a deflation with ${\rm CoCone}(f)\in{\rm T}\mathcal{F}\}$.
\end{center}

$(5)$ $(\mathcal{C}, {\rm T}\mathcal{F})$ is a complete cotorsion pair, ${\rm CoFib}\subseteq \{f\mid f$ is an inflation with ${\rm Cone}(f)\in\mathcal{C}\}$, and ${\rm TFib}\subseteq \{f\mid f$ is a deflation with ${\rm CoCone}(f)\in{\rm T}\mathcal{F}\}$.
\end{prop}
\begin{proof} $(1)\Rightarrow (2)$. This follows from Lemma \ref{lem4.3}(1).

$(2)\Rightarrow (1)$. Let $i: A\rightarrow B$ be an inflation with ${\rm Cone}(f)\in\mathcal{C}$. For any $p\in{\rm TFib}$, then $p$ is a deflation with ${\rm CoCone}(p)\in{\rm T}\mathcal{F}$ by assumption. Since $\mathbb{E}(\mathcal{C}, {\rm T}\mathcal{F})=0$, $i$ has the left lifting property with respect to $p$ by Lemma \ref{lem3}. Thus $i$ is a cofibration by Proposition \ref{prop2.22}(1). Thus $\{f\mid f$ is an inflation with ${\rm Cone}(f)\in\mathcal{C}\}\subseteq{\rm CoFib}$, and ${\rm CoFib}=\{f\mid f$ is an inflation with ${\rm Cone}(f)\in\mathcal{C}\}$ by assumption.

 $(1)\Leftrightarrow (3)$  can be proved similarly.

  $(4)\Rightarrow (1)$ is clear; and  $(1)\Rightarrow (4)$ is also clear by (1) and (3).

   $(5)\Rightarrow (2)$ is clear.

   $(2)\Rightarrow (5)$. For any object $A\in\mathcal{B}$, there exists a right $\mathcal{C}$-approximation $f: C\rightarrow A$ such that $f\in{\rm TFib}$ by Lemma \ref{lem4.2}(1). Hence there is an $\mathbb{E}$-triangle $\xymatrix{Y\ar[r]&C\ar[r]^f&A\ar@{-->}[r]&}$ with $Y\in{\rm T}\mathcal{F}$ by assumption. Similarly, one has an $\mathbb{E}$-triangle $\xymatrix{A\ar[r]^i&Y'\ar[r]&C'\ar@{-->}[r]&}$ with $Y'\in{\rm T}\mathcal{F}$ and $C'\in\mathcal{C}$. Hence $(\mathcal{C}, {\rm T}\mathcal{F})$ is a complete cotorsion pair by Fact \ref{fact2.18}.
\end{proof}

\begin{definition}{\rm A model structure on an extriangulated category $\mathcal{B}$ is {\it weakly projective} provided that each object is fibrant and it satisfies the equivalent conditions in Proposition \ref{prop4.4}.}
\end{definition}

\subsection{Proof of Theorem \ref{thm4.6}}

 By Theorem \ref{thmA}, ${\rm Im}\Phi\in\Gamma$ and $\Psi\Phi={\rm Id}_\Omega$. It remains to prove that ${\rm Im}\Psi\in\Omega$ and $\Phi\Psi={\rm Id}_\Gamma$.

Let $({\rm CoFib}, {\rm Fib}, {\rm Weq})$ be a weakly projective model structure on $\mathcal{B}$. It follows from Proposition \ref{prop4.4}(5) that $(\mathcal{C}, {\rm T}\mathcal{F})$ is a complete cotorsion pair. Note that $\mathcal{F}=\mathcal{B}$ by definition. Thus

\begin{center} $\mathcal{C}\cap {\rm T}\mathcal{F}=\mathcal{C}\cap \mathcal{F}\cap\mathcal{W}=\mathcal{C}\cap\mathcal{W}={\rm T}\mathcal{C}$.\end{center}

\noindent So $\mathcal{C}\cap{\rm T}\mathcal{F}={\rm T}\mathcal{C}$ is contravariantly finite in $\mathcal{B}$ by Lemma \ref{lem4.2}(1).

We need to prove that the cotorsion pair $(\mathcal{C}, {\rm T}\mathcal{F})$ is hereditary (and hence  $(\mathcal{C}, {\rm T}\mathcal{F})\in\Gamma$), and that
$({\rm CoFib}, {\rm Fib}, {\rm Weq})=({\rm CoFib}_\omega, {\rm Fib}_\omega, {\rm Weq}_\omega)$, where $\omega=\mathcal{C}\cap{\rm T}\mathcal{F}={\rm T}\mathcal{C}$. We divide the proof into several steps.

Since $({\rm CoFib}, {\rm Fib}, {\rm Weq})$ is a weakly projective model structure, by Proposition \ref{prop4.4}(4) one has

\begin{center}${\rm CoFib}=\{f\mid f$ is an inflation with ${\rm Cone}(f)\in\mathcal{C}\}={\rm CoFib}_\omega$\end{center}
and
\begin{center}${\rm TFib}=\{f\mid f$ is a deflation with ${\rm CoCone}(f)\in{\rm T}\mathcal{F}\}={\rm TFib}_\omega$.\end{center}

{\bf Claim 1:} ${\rm TCoFib}=\{f\mid f$ is a splitting inflation with ${\rm Cone}(f)\in{\rm T}\mathcal{C}\}={\rm TCoFib}_\omega$.

Let $f: A\rightarrow B$ be a splitting inflation with ${\rm Cone}(f)\in{\rm T}\mathcal{C}$. Then there is an isomorphism ${\tiny\begin{pmatrix}f\\i\end{pmatrix}}: A\oplus {\rm Cone}(f)\rightarrow B$. It is clear that the square
\begin{center}$\xymatrix@R=0.5cm{0\ar[r]\ar[d]&A\ar[d]^f\\{\rm Cone}(f)\ar[r]^-{i}&B}$
\end{center}
is a push-out. Since ${\rm Cone}(f)$ is a trivially cofibrant object, $0\rightarrow {\rm Cone}(f)\in{\rm TCoFib}$, and $f\in {\rm TCoFib}$ by Fact \ref{fact2.21}(3).

Conversely, let $f: A\rightarrow B$ be a trivial cofibration. Then $f\in{\rm CoFib}$, and hence $f$ is an inflation with ${\rm Cone}(f)\in\mathcal{C}$ by Proposition \ref{prop4.4}, so there is an $\mathbb{E}$-triangle $\xymatrix{A\ar[r]^f&B\ar[r]^{\pi\ \ \ }&{\rm Cone}(f)\ar@{-->}[r]&}$. Since $A\in\mathcal{B}=\mathcal{F}$, $1_A: A\rightarrow A$ factors through $f$ by Lemma \ref{lem4.1}(2). Therefore, $f$ is a splitting inflation, which implies that $\pi$ is an epimorphism. It is easy to prove that the following square
\begin{center}$\xymatrix@R=0.5cm{A\ar[r]\ar[d]_f&0\ar[d]\\B\ar[r]^-{\pi}&{\rm Cone}(f)}$
\end{center}
is a push-out. Since ${\rm Cone}(f)$ is a trivially cofibration, $0\rightarrow {\rm Cone}(f)\in{\rm TCoFib}$ is also a trivial cofibration by Fact \ref{fact2.21}(3), i.e., ${\rm Cone}(f)\in{\rm T}\mathcal{C}$. This completes the proof of {\bf Claim 1}.

{\bf Claim 2:} ${\rm Weq}={\rm Weq}_\omega$. This follow from ${\rm Weq}={\rm TFib}\circ{\rm TCoFib}={\rm TFib}_\omega\circ{\rm TCoFib}_\omega={\rm Weq}_\omega$.

{\bf Claim 3:} ${\rm Fib}=\{f:A\rightarrow B\mid {\rm Hom}_\mathcal{B}(W, f)$ is epic for any $W\in\omega\}={\rm Fib}_\omega$.

If $g: C\rightarrow D$ is a fibration, then $g$ has the right lifting property with respect to all trivial cofibrations by Proposition \ref{prop2.22}(3). For any $W\in\omega$, $0\rightarrow W$ is a trivial cofibration by {\bf Claim 1}. Consider the commutative square
\begin{center}$\xymatrix@R=0.5cm{0\ar[r]\ar[d]&C\ar[d]^f\\W\ar@{-->}[ur]^h\ar[r]_{g\ \ }&D,}$
\end{center}
there is a morphism $h: W\rightarrow C$ such that $f=gh$ by the lifting axiom. That is, $\mathcal{B}(W, g)$ is epic, and $g\in{\rm Fib}_\omega$.

Conversely, suppose that $g:C\rightarrow D$ is a morphism such that $\mathcal{B}(W, g)$ is epic for any $W\in\omega$, we claim $g\in{\rm Fib}$. It suffices to prove that $g$ has the right lifting property with respect to all trivial cofibrations by Proposition \ref{prop2.22}(3). It follows from {\bf Claim 1} that any trivial cofibration can be written as ${\tiny\begin{pmatrix}1\\0\end{pmatrix}}: A\rightarrow A\oplus W$ with $W\in\omega$. Consider the commutative square
\begin{center}$\xymatrix{A\ar[r]^s\ar[d]_{\tiny\begin{pmatrix}1\\0\end{pmatrix}}&C\ar[d]^g\\A\oplus W\ar@{-->}[ur]^-{(s, p)}\ar[r]_{(f,t)}&D.}$
\end{center}
There is a morphism $p: W\rightarrow C$ such that $t=gp$ because $(W, g)$ is epic for any $W\in\omega$. It is easy to check that the morphism $(s, p): A\oplus W\rightarrow C$ satisfies $(s, p){\tiny\begin{pmatrix}1\\0\end{pmatrix}}=s$ and $g(s,p)=(f, t)$. This completes the proof of {\bf Claim 3}.

Now we have proved that ${\rm CoFib}={\rm CoFib}_\omega, {\rm Fib}={\rm Fib}_\omega$ and ${\rm Weq}={\rm Weq}_\omega$. This implies that $({\rm CoFib}_\omega, {\rm Fib}_\omega, {\rm Weq}_\omega)$ is also a model structure. It follows from Theorem \ref{thmC} that $(\mathcal{C}, {\rm T}\mathcal{F})$ is hereditary. Thus ${\rm Im}\Psi\in\Omega$ and $\Phi\Psi={\rm Id}_\Gamma$. This completes the proof.\hfill$\Box$

\section{Applications}
 Throughout this section we fix an extriangulated category $\mathcal{B}$. We begin this section with the following definitions.

\begin{definition}{\rm  (\cite[Definition 4.1]{AT}) \ Let $\mathcal B$ be an extriangulated category and $\X$ a subcategory of $\mathcal{B}$.
For each $n\geq 0$, we inductively define subcategories
$\mathcal X_n^{\wedge}$ and $\mathcal X_n^{\vee}$ of $\mathcal C$
as $\mathcal X_n^{\wedge}:={\rm Cone}(\mathcal X_{n-1}^{\wedge},\mathcal X)$
and $\mathcal X_n^{\vee}:={\rm CoCone}(\mathcal X,\mathcal X_{n-1}^{\vee})$,
where $\mathcal X_{-1}^{\wedge}:=\{0\}$
and $\mathcal X_{-1}^{\vee}:=\{0\}$. Put
$$\mathcal X^{\wedge}:=\bigcup\limits_{n\geq 0}\mathcal X_n^{\wedge}~~\mbox{and}~~
\mathcal X^{\vee}:=\bigcup\limits_{n\geq 0}\mathcal X_n^{\vee}.$$}
\end{definition}

Let $\mathcal{X}$ be a subcategory of $\mathcal{B}$. Recall from \cite{Z} that $\mathcal{X}$ is called a \emph{thick subcategory} of $\mathcal{B}$ if it is closed under extensions, cones of inflations, cocones of deflations and direct summands. Denote by ${\rm add} \mathcal{X}$ the smallest subcategory of
$\mathcal{B}$ containing $\mathcal{X}$ and closed under finite direct sums and direct summands.

\begin{definition}{\rm (\cite[Definition 5.1]{AT})
Let $\mathcal{B}$ be an extriangulated category.
A subcategory $\mathcal M$ of $\mathcal{B}$ is called a \emph{silting subcategory} if
it satisfies the following conditions.

{\rm (a)} $\mathcal M$ is a presilting subcategory, that is, $\mathbb{E}^k(\mathcal M,\mathcal M)=0$ for all $k\geq 1$.

{\rm (b)}  $\mathcal{B}$ is the smallest thick subcategory containing
$\mathcal M$.

We say that an object $M\in\mathcal{B}$ is \emph{silting} if so is ${\rm add} M$.}
\end{definition}

We give some examples of silting subcategories.

\begin{ex}\label{ex:4.1} This example comes from \cite[Example 2.2]{AI}.

(1) Let $A$ be a ring. Then $A$ regarded as a stalk complex is a tilting object
in $K^b({\rm proj}\emph{-}A)$. More generally, any tilting complex of $A$
is a tilting object in $K^b({\rm proj}\emph{-}A)$.
By definition, any tilting object is a silting object in triangulated categories.

(2)  Let $A$ be a differential graded ring and let $D(A)$ be the derived category of $A$.
If $H^i(A)=0$ for any $i>0$,
then the thick subcategory ${\rm thick}(A)$ of $D(A)$ generated by $A$ has a silting
object $A$.
\end{ex}

\begin{ex}\label{ex:4.2}
Let $\Lambda$ be an artin algebra. We denote by $\mathcal P^{\infty}(\Lambda)$
the category of finitely
generated right $\Lambda$-modules of finite projective dimension. Since $\mathcal P^{\infty}(\Lambda)$
is closed under extensions, it is an extriangulated category. By \cite[Example 5.2]{AT},
we know that ${\rm add}\Lambda$ is a silting
subcategory of $\mathcal P^{\infty}(\Lambda)$.
Moreover, ${\rm add}\Lambda$ is functorially finite.
\end{ex}

\begin{ex}\label{ex:4.3}
\begin{itemize}
\item[(1)] Let $\Lambda$ be a finite dimensional $k$-algebra and let mod$\Lambda$ denote the category of finitely
generated right $\Lambda$-modules. It follows from \cite[Corollary 5.6]{AT} that
the global dimension of $\Lambda$ is finite if and only if tilting $\Lambda$-modules in ${\rm mod}\Lambda$ coincide with silting objects of ${\rm mod}\Lambda$.

\item[(2)] Let $Q$ be the following infinite quiver:
$$\xymatrix@C=0.5cm@R0.5cm{\cdots \ar[r]^{x_{-5}} &-4 \ar[r]^{x_{-4}} &-3 \ar[r]^{x_{-3}} &-2 \ar[r]^{x_{-2}} &-1 \ar[r]^{x_{-1}} &0 \ar[r]^{x_{0}} &1 \ar[r]^{x_1} &2 \ar[r]^{x_2} &3 \ar[r]^{x_3} &4 \ar[r]^{x_4} &\cdots}$$
Let $\Lambda=kQ/[x_ix_{i+1}x_{i+2},x_{4k}x_{4k+1}](i\neq 4k)$. Then the Auslander-Reiten quiver of ${\rm mod}\Lambda$ is the following.
$$
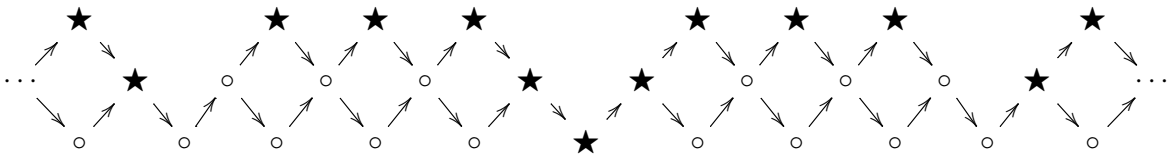
\xymatrix@C=0.1cm@R0.2cm{
 &\bigstar \ar[dr] &&&&\bigstar\ar[dr] &&\bigstar \ar[dr] &&\bigstar \ar[dr] &&&&\bigstar \ar[dr] &&\bigstar \ar[dr] &&\bigstar \ar[dr] &&&&\bigstar \ar[dr]\\
\cdots \ar[ur] \ar[dr] &&\bigstar \ar[dr] &&\circ \ar[ur] \ar[dr] &&\circ \ar[ur] \ar[dr] &&\circ \ar[ur] \ar[dr] &&\bigstar  \ar[dr] &&\bigstar \ar[ur] \ar[dr] &&\circ \ar[ur] \ar[dr] &&\circ \ar[ur] \ar[dr] &&\circ  \ar[dr] &&\bigstar \ar[ur] \ar[dr] &&\cdots\\
 &\circ \ar[ur] &&\circ \ar[ur] &&\circ\ar[ur] &&\circ \ar[ur] &&\circ \ar[ur] &&\bigstar \ar[ur] &&\circ \ar[ur] &&\circ\ar[ur] &&\circ \ar[ur]  &&\circ \ar[ur] &&\circ \ar[ur]
}
$$
The indecomposable objects denoted by $\bigstar$ are the indecomposable objects of a silting
 subcategory in ${\rm mod}\Lambda$.
\end{itemize}
\end{ex}

Recall from \cite{AT} that a complete cotorsion pair $(\mathcal{X}, \mathcal{Y})$ in $\mathcal{B}$ is {\it bounded} if $\mathcal{B}=\mathcal{X}^{\wedge}$ and $\mathcal{B}=\mathcal{Y}^{\vee}$. Adachi and Tsukamoto \cite{AT} proved that any silting object in an extriangulated category can induce a bounded hereditary cotorsion pair.

\begin{lem}\label{main}{\rm (\cite[Theorem 5.7]{AT})}
Let $\mathcal{B}$ be an extriangulated category. Then there exists a
bijection between the following sets:

$(1)$ \  The set of complete cotorsion pairs $(\mathcal X, \mathcal Y)$ satisfying that $\mathcal{B}=\mathcal X^{\wedge}=\mathcal Y^{\vee}$, and $\mathbb{E}^{i}(X, Y)=0$ for each $i\geqslant2$ and all $X\in{\mathcal{X}}$ and $Y\in{\mathcal{Y}}$.

$(2)$ \  The set of silting subcategories of $\mathcal{B}$.

In particular, if $\mathcal M$ is a silting subcategory of $\mathcal{B}$,
then $(\mathcal M^{\vee},\mathcal M^{\wedge})$ is a bounded hereditary cotorsion pair
in $\mathcal{B}$ satisfying $\mathcal M=\mathcal M^{\vee}\cap\mathcal M^{\wedge}$.
\end{lem}

As a direct consequence of Theorem \ref{thmA} and Lemma \ref{main} one has
\begin{cor}\label{cor:4.2} Let $\mathcal B$ be a weakly idempotent complete extriangulated category. Then any silting object $M$ of $\mathcal B$ induces the model structure $({\rm CoFib}_{\mathcal M}, {\rm Fib}_{\mathcal M}, {\rm Weq}_{\mathcal M})$ on $\mathcal{B}$, where $\mathcal M={\rm add}M$ is a silting subcategory of $\mathcal B$.
\end{cor}

We recall the following notion of co-$t$-structures, which was independently introduced by Bondarko \cite{Bon} and Pauksztello \cite{Pau} as an analog of $t$-structures defined in \cite{B}.

\begin{definition}  {\rm (\cite{Bon,Pau}) Let $\mathcal{B}$ be a triangulated category with the shift functor [1]. A \emph{co-t-structure} on $\mathcal{B}$ is a pair $(\mathcal{X},\mathcal{Y})$ of subcategories of $\mathcal{B}$ such that
\begin{enumerate}
\item $\mathcal{X}[-1]\subseteq \mathcal{X}$ and $\mathcal{Y}[1]\subseteq \mathcal{Y}$.

\item $\mathrm{Hom}_{\mathcal{A}}(\mathcal{X}[-1],\mathcal{Y})=0$.

\item Any object $T\in{\mathcal{B}}$ has a distinguished triangle
$\xymatrix@C=0.6cm{X\ar[r]&T\ar[r]&Y\ar[r]&X[1]}$
\noindent in $\mathcal{B}$ with $X\in{\mathcal{X}[-1]}$ and $Y\in{\mathcal{Y}}$.
\end{enumerate}}
{\rm In this case, the intersection $\mathcal{X}\cap\mathcal{Y}$ will be called the \emph{coheart} of
the co-t-structure $(\mathcal{X},\mathcal{Y})$.}
\end{definition}

Note that hereditary cotorsion pairs on a triangulated category are exactly co-t-structures (see \cite[Example 3.4]{AT}). In combination this with Theorem \ref{thm4.6}, we have the following result.

\begin{cor}\label{corB} Let $\mathcal{X}$ and $\mathcal{Y}$ full additive subcategories of a triangulated category $\mathcal{B}$ which are closed under direct summands and isomorphisms. Then there exists a one-to-one correspondence between weakly projective model structures on $\mathcal{B}$ and co-t-structures with the coheart contravariantly finite in $\mathcal{B}$.
\end{cor}

\section*{Acknowledgements}
We sincerely thank the anonymous referee for valuable comments and suggestions, which improve the manuscript.

\vspace{0.5cm}
\hspace{-4mm}\textbf{Data Availability}\hspace{2mm} Data sharing not applicable to this article as no datasets were generated or analysed during
the current study.
\vspace{2mm}

\hspace{-4mm}\textbf{Conflict of Interests}\hspace{2mm} The authors declare that they have no conflicts of interest to this work.


\renewcommand\refname
\end{document}